\documentclass[12p,a4paper]{article}
\usepackage[english]{babel}
\usepackage{amssymb}
\usepackage{amsmath}
\usepackage{amsthm}
\usepackage{changepage}
\usepackage{color}
\usepackage{enumerate}
\usepackage{comment}
\usepackage{hyperref}

\usepackage[margin=1.2in]{geometry}
\usepackage[color=white]{todonotes}

\newtheorem{theorem}{Theorem}[section]
\newtheorem{lemma}[theorem]{Lemma}
\newtheorem{proposition}[theorem]{Proposition}
\newtheorem{definition}[theorem]{Definition}
\newtheorem{remark}[theorem]{Remark}

\newtheorem{corollary}[theorem]{Corollary}

\title{Polynomial \(D(4)\)-quadruples over Gaussian Integers}

\author{Marija Bliznac Trebje\v{s}anin \and Sanda Buja\v{c}i\'{c} Babi\'{c}}
\date{}

\begin{document}
\maketitle

\abstract{A set $\{a, b, c, d\}$ of four non-zero distinct polynomials in $\mathbb{Z}[i][X]$ is said to be a Diophantine $D(4)$-quadruple if the product of any two of its distinct elements increased by 4 is a square of some polynomial in $\mathbb{Z}[i][X]$.
In this paper we prove that every $D(4)$-quadruple in $\mathbb{Z}[i][X]$ is regular, or equivalently that the equation $$(a+b-c-d)^2=(ab+4)(cd+4)$$ holds for every $D(4)$-quadruple in $\mathbb{Z}[i][X]$.}

\medskip





\section{Introduction}\label{intro}

\noindent Let $n$ be a nonzero integer. A set of $m$ distinct positive integers $\{a_1, a_2, \dots, a_m\}$ is called a Diophantine $m$-tuple with property $D(n)$, or simply a Diophantine $D(n)$-$m$-tuple, if $$a_i a_j +n$$ is a perfect square for all $1\leq i < j \leq m$. 

The Diophantine $D(1)$-$m$-tuple is simply called Diophantine $m$-tuple. Diophantus of Alexandria was the first  mathematician to deal with the problem of finding a set consisting of four distinct positive rational numbers such that the product of any two of them increased by $1$ is a square. The Diophantine quadruple $\{1, 3, 8, 120\}$ found by Fermat was the first such set of integers.  The most studied case is for $n=1$, but besides this case, the cases $n=-1$ and $n=4$ have also been studied in recent years. It is proved that there is no $D(-1)$-quadruple \cite{cipu}. It is useful to point out that the non-existence of $D(-1)$-quadruples implies that there are no $D(-4)$-quadruples \cite{duj}, too. Moreover, similar conjectures and observations can be made for the $n=4$ case if they hold for the $n=1$ case. 

A $D(4)$-pair $\{a,b\}$ can be extended with a larger element $c$ to form a $D(4)$-triple. The smallest such $c$ has the form $c=a+b+2r$, for $r=\sqrt{ab+4}$, and such a triple is often called a regular triple. It is easy to note that there are infinitely many extensions of a $D(4)$-pair to a $D(4)$-triple, and they can be studied by finding solutions to a Pellian equation $$bs^2-at^2=4(b-a),$$ where $s, t$ are positive integers given by $ac+4 = s^2, \ bc + 4 = t^2$, respectively.

One way of generalizing the presented concept is to introduce a polynomial $m$-tuple with property $D(n)$ or a polynomial $D(n)$-$m$-tuple. In this paper, we deal with the polynomial variant of the introduced problem for $n=4$ and polynomials in $\mathbb{Z}[i][X]$.
A similar problem was first studied by Jones for the set of the polynomials with integer coefficients and $n=1$ \cite{jones1, jones2}. 

For a start, we present some basic definitions, remarks and results that we use in our work.

\begin{definition}
Let $m\geq 2$ and let $R$ be a commutative ring with unity. Let $n\in R$ be a non-zero element and $\{a_1,a_2,\dots,a_m\}$ a set of $m$ distinct non-zero elements in $R$ such that $a_ia_j+n$ is a square of an element in $R$ for $1\leq i<j\leq m$. The set $\{a_1,a_2,\dots,a_m\}$ is called a Diophantine $m$-tuple with the property $D(n)$ or simply a $D(n)$-$m$-tuple in $R$.
\end{definition}

In the case where $R$ is a polynomial ring and $n$ is a constant polynomial, it is usually assumed that not all polynomials in such a $D(n)$-tuple are constant.

Let $\{a,b,c\}$ be a $D(4)$-triple in $\mathbb{Z}[i][X]$ such that 
\begin{equation}\label{jdbe_osnovne}
    ab+4=r^2,\ ac+4=s^2,\ bc+4=t^2,
\end{equation}
where $r,s,t\in \mathbb{Z}[i][X]$. 

\begin{definition}\label{expc}
A $D(4)$-triple $\{a,b,c\}$ in $\mathbb{Z}[i][X]$ is called regular if
$$(c-b-a)^2=4(ab+4).$$
Or, more explicitly, if
\begin{align}
    c&=c_{\pm}=a+b\pm 2r,\label{jdba_c_regularni}\\
    ac_{\pm}+4&=(a\pm r)^2,\ bc_{\pm}+4=(b\pm r)^2.\label{jdba_c_regularni1}
\end{align}
\end{definition}


As said before, besides regular extensions of a pair to a triple, it is not hard to find non-regular triples.

Similarly, there are extensions of a $D(4)$-triple $\{a,b,c\}$ to a $D(4)$-quadruple $\{a,b,c,d\}$ given by an explicit expression. 
\begin{definition}
A $D(4)$-quadruple $\{a,b,c,d\}$ in $\mathbb{Z}[i][X]$ is called regular if 
$$(a+b-c-d)^2=(ab+4)(cd+4),$$
or equivalently if 
\begin{equation}\label{jdba_d_pm}
    d=d_{\pm}=a+b+c+\frac{1}{2}(abc\pm rst).
\end{equation}
In this case, we have 
\begin{align}\label{uvw}
&ad_{\pm}+4 = \left(\frac{rs \pm at}{2}\right)^2 = u_{\pm}^2, \ \ bd_{\pm}+4 =\left( \frac{rt \pm bs}{2}\right)^2=v_{\pm}^2,\\  &cd_{\pm}+4 = \left(\frac{st \pm cr}{2}\right)^2=w_{\pm}^2.\nonumber
\end{align}

An irregular $D(4)$-quadruple in $\mathbb{Z}[i][X]$ is one that is not regular.
\end{definition}

We denote by $d_+$ the polynomial with the higher degree and by $d_-$ the polynomial with a lower degree among the two polynomials $d_{\pm}$. 

It is easy to check that the polynomials $abc \pm rst$ are divisible by $2$ in $\mathbb{Z}[i][X]$ to see that such polynomials exist and are well-defined by equation (\ref{jdba_d_pm}) for any $D(4)$-triple $\{a,b,c\}$.

\begin{remark}
There always exist regular $D(4)$-quadruples $\{a, b, c, d_{\pm}\}$, where $d_{\pm}$ is defined by (\ref{jdba_d_pm}).
Specifically, any $D(4)$-pair $\{a,b\}$ in $\mathbb{Z}[i][X]$ can be extended to regular $D(4)$-quadruples
\begin{equation}\label{jdba_par_do_cetvorke}
\{a,b,a+b\pm 2r,r(r\pm a)(b\pm r) \}.\end{equation}
\end{remark}
\medskip

If $\{a,b\}$ is a $D(4)$-pair in $\mathbb{C}[X]$ then $\{\frac{a}{2},\frac{b}{2}\}$ is a $D(1)$-pair in $\mathbb{C}[X]$. An important result now follows from \cite[Lemma 1]{dl_17}.

\begin{lemma}\label{lema_samo_jedan_konstantan}
Let $\{a_1,a_2,\dots,a_m\}$ be a $D(4)$-$m$-tuple  in $\mathbb{C}[X]$, such that not all $a_i$'s are constant polynomials. Then $a_i\neq a_j$ for $i\neq j$ and at most one of the polynomials $a_i$, $i=1,\dots,m$, is constant.  
\end{lemma}

The main result of our work is the following theorem. 

\begin{theorem}\label{tm_glavni}
Every $D(4)$-quadruple in $\mathbb{Z}[i][X]$ is regular. 
\end{theorem}

To prove Theorem \ref{tm_glavni}, we use some methods introduced in \cite{glavni} and \cite{dl_17}, but some different approaches and strategies were needed to prove the main theorem in all possible cases. Since this topic has not been extensively studied for the case $n=4$, we also proved and joined some essential results from \cite{dujjur} and \cite{dl_17}.

As usual, we first deal with the system of simultaneous Pellian equations and then find the intersection of the binary recurrent sequences thus generated. The proof is done by using congruence relations and the gap principle.


A related problem of the regularity of the polynomial $D(-4;4)$-quadruple over $\mathbb{Z}[X]$ can be considered because Theorem \ref{tm_glavni} allows us to prove the analogous result of \cite{bfj} without additional conditions. Before we can highlight this important consequence of our main result, we need the following definition. 

\begin{definition}
A set $\{a,b,c,d\}$ of four non-zero distinct polynomials in $\mathbb{Z}[X]$ is said to have $D(-4;4)$ property, or that it is a polynomial $D(-4;4)$-quadruple if $\{a,b,c\}$ is a $D(-4)$-triple in $\mathbb{Z}[X]$ and $$ad+4, \quad bd+4, \quad  cd+4$$ are all squares of some polynomials in $\mathbb{Z}[X]$. 
\end{definition}

An analogous problem on $D(-1;1)$-quadruples was considered in \cite{bfj} and in \cite{glavni} the authors showed that the assertion about the regularity of $D(-1;1)$-quadruples follows as a corollary of their main result. We carry out the corresponding result in our case.

\begin{corollary}\label{kor1}
Every polynomial $D(-4;4)$-quadruple in $\mathbb{Z}[X]$ is regular. More precisely, any $D(-4)$-triple $\{a,b,c\}$ in $\mathbb{Z}[X]$ can be extended to a $D(-4;4)$-quadruple $\{a,b,c,d\}$ in $\mathbb{Z}[X]$ only with
$$d=d_{\pm}=-(a+b+c)+\frac{1}{2}(abc\pm r's't'),$$
where $(r')^2=ab-4$, $(s')^2=ac-4$ and $(t')^2=bc-4$.
\end{corollary}
\begin{proof}
Let $\{a, b, c\}$ be a $D(-4)$-triple in $\mathbb{Z}[X]$. It is easy to see that in that case $\{ai,bi,ci\}$ is a $D(4)$-triple in $\mathbb{Z}[i][X]$. If there exists $d \in \mathbb{Z}[X]$ such that $ad+4=x^2$, for some $x\in\mathbb{Z}[X]$, then $(ai)\cdot (-di)+4=x^2$. So, $\{ai,bi,ci, -di\}$ is a $D(4)$-quadruple in $\mathbb{Z}[i][X]$. By Theorem \ref{tm_glavni} we must have
$$-di=ai+bi+ci+\frac{1}{2}(ai\cdot bi\cdot ci\pm r'i\cdot s'i\cdot t'i),$$
which yields
$
d=-(a+b+c)+\frac{1}{2}(abc\pm r's't').\hfill\qedhere
$
\end{proof}

\section{General properties and Pellian equations} 
\label{general_pell}
Let us deal with an arbitrary extension of a $D(4)$-triple $\{a, b, c\}$ in $\mathbb{Z}[i][X]$ to a $D(4)$-quadruple $\{a, b, c, d\}$ in $\mathbb{Z}[i][X]$. 
For a $D(4)$-quadruple $\{a,b,c,d\}$, there exist polynomials $x, y, z\in\mathbb{Z}[i][X]$ such that
\begin{equation}\label{jdb_jednakosti_za_d}
ad+4=x^2,\ bd+4=y^2,\ cd+4=z^2.
\end{equation}

Eliminating $d$ from (\ref{jdb_jednakosti_za_d}), we obtain the system of simultaneous Pellian equations
\begin{align}
    az^2-cx^2&=4(a-c),\label{jdba_pellova_prva}\\
    bz^2-cy^2&=4(b-c).\label{jdba_pellova_druga}
\end{align}

Let us denote by $\alpha,\ \beta,\ \gamma$ the degrees of the polynomials $a,\ b,\ c$, respectively. Without loss of generality, we will assume $0\leq \alpha \leq \beta \leq \gamma$ and $\beta,\gamma>0$, since by Lemma \ref{lema_samo_jedan_konstantan} only $a$ can be a constant polynomial. 

A modified version of \cite[Lemma 2.1]{fj_19} appears in the following Lemma \ref{lema_rj_pellove_jedn}. More precisely, after dealing with certain algebraic transformations for each situation, these two lemmas follow directly from \cite[Lemma 4]{dujjur}.

\begin{lemma}\label{lema_rj_pellove_jedn}
There exist solutions $(z_0,x_0)$ and $(z_1,y_1)$, $z_0,z_1,x_0,y_1\in\mathbb{Z}[i][X]$, of (\ref{jdba_pellova_prva}) and (\ref{jdba_pellova_druga}), respectively, such that
\begin{enumerate}[i)]
    \item \begin{equation}\label{deg1}    \deg(z_0)\leq \frac{3\gamma-\alpha}{4},   \quad \deg(x_0)\leq \frac{\alpha+\gamma}{4},
\end{equation}
\begin{equation}\label{ineq20}
    \deg(z_1)\leq \frac{3\gamma-\beta}{4}, \quad \deg(y_1)\leq \frac{\beta+\gamma}{4}. 
\end{equation}

\item There exist non-negative integers $m$ and $n$ such that
\begin{equation}
        \label{jdba_opce_rjesenje_PRVE}
        z\sqrt{a}+x\sqrt{c}=(z_0\sqrt{a}+x_0\sqrt{c})\left(\frac{s+\sqrt{ac}}{2}\right)^{2m},
\end{equation}
\begin{equation}
        \label{jdba_opce_rjesenje_DRUGE}
        z\sqrt{b}+y\sqrt{c}=(z_1\sqrt{b}+y_1\sqrt{c})\left(\frac{t+\sqrt{bc}}{2}\right)^{2n}.
\end{equation}
\end{enumerate}
\end{lemma}

According to Lemma \ref{lema_rj_pellove_jedn}, there exist $d_0, d_1\in\mathbb{Z}[i][X]$ such that 
\begin{equation}\label{d0}
    ad_0+4 = x_0^2 \ \ \textrm{and} \ \ cd_0+4 = z_0^2,
\end{equation}
and
\begin{equation}\label{d1}
    bd_1+4 = y_1^2 \ \ \textrm{and} \ \ cd_1+4 = z_1^2.
\end{equation}

From (\ref{jdba_opce_rjesenje_PRVE}) we generally have 
$$z_{m+1}\sqrt{a}+x_{m+1}\sqrt{c} = (z_m\sqrt{a}+x_m\sqrt{c})\frac{ac-2+s\sqrt{ac}}{2}, \ \ m\geq 0.$$

After some elementary transformations 
$$z_{m+1} = \frac{z_m(ac-2)+x_mcs}{2}, \ \ x_{m+1} = \frac{x_m(ac-2)+z_mas}{2}$$
are obtained. For
$z=v_m=w_n$ for some $m,n\geq 0$, binary recurrence sequences $(v_m)_{m\geq0}$ and $(w_n)_{n\geq0}$ are  given by
\begin{align}
    &v_0=z_0,\quad v_1=\frac{1}{2}(sz_0+cx_0),\quad v_{m+2}=sv_{m+1}-v_m,\label{rekurzija_vm}\\
    &w_0=z_1,\quad w_1=\frac{1}{2}(tz_1+cy_1),\quad w_{n+2}=tw_{n+1}-w_n.\label{rekurzija_wn}
\end{align}

Because we are dealing with polynomials over Gaussian integers, it remains to check if $v_1, w_1 \in\mathbb{Z}[i][X]$.
By (\ref{jdbe_osnovne}) and (\ref{d0}), we get that $4|(z_0(ac-2)+x_0cs)(z_0(ac-2)-x_0cs)$ and then we conclude that $2$ divides each of the factors, so $v_1 \in\mathbb{Z}[i][X]$. The proof is analogous for $w_1 \in\mathbb{Z}[i][X]$.

Since the proofs of the next two lemmas are similar to those in \cite[Lemma 5]{dl_17}, we omit them.

\begin{lemma}\label{kongruencije_rj}Assume that $(z, x), (z', y')$ are solutions of (\ref{jdba_pellova_prva}) and (\ref{jdba_pellova_druga}), respectively. If $z^2 \equiv 4 \pmod{c}$, then $x^2 \equiv 4 \pmod{a}$ and the same is satisfied for $(z_0, x_0)$ solution of (\ref{jdba_pellova_prva}). If $z'^2 \equiv 4 \pmod{c}$ holds, then $y'^2 \equiv 4 \pmod{b}$ and the same is satisfied for $(z_1, y_1)$ solution of (\ref{jdba_pellova_druga}).  Additionally, if $z_0, z_1$ are not constant polynomials, then $\deg(z_0), \deg(z_1) \geq \frac{\gamma}{2}$.
\end{lemma}

\begin{lemma}\label{lema_stupnjevi}
Let $\{a,b,c,d\}$ be a $D(4)$-quadruple in $\mathbb{Z}[i][X]$ and $(v_m)_{m\geq0}$ and $(w_n)_{n\geq0}$ be sequences defined as in (\ref{rekurzija_vm}) and (\ref{rekurzija_wn}), respectively. Then for $m\geq 1$ and $n\geq 1$ the following holds
\begin{equation}\label{degvm}
    \deg(v_m)=(m-1)\frac{\alpha+\gamma}{2}+\deg(v_1),
\end{equation}
\begin{equation}\label{degwn}
    \deg(w_n)=(n-1)\frac{\beta+\gamma}{2}+\deg(w_1).
\end{equation}
Also, 
\begin{equation}\label{degv1}
\frac{\gamma}{2}\leq \deg(v_1)\leq \frac{\alpha+5\gamma}{4},
\end{equation}
\begin{equation}\label{degw1}
    \frac{\gamma}{2}\leq \deg(w_1)\leq \frac{\beta+5\gamma}{4}.
\end{equation}
\end{lemma}

\noindent The proof of the next lemma is conducted similarly as in \cite[Lemma 1]{dujjur}.

\begin{lemma}\label{degineq}
For $v_m = w_n$, where $(v_m)_{m\geq0}$ and $(w_n)_{n\geq0}$ are defined by (\ref{rekurzija_vm}) and (\ref{rekurzija_wn}), we have $$n-1 \leq m \leq 2n+1.$$
\end{lemma}

Now we will state some congruence relations that hold for $v_{m}$ and $w_{n}$ and which will be essential for the final proof of Theorem \ref{tm_glavni}.

The following lemma is easily proved by induction so the proof is omitted.

\begin{lemma}\label{initial1}
The sequences $(v_m)_{m\geq0}$ and $(w_n)_{n\geq0}$, given by (\ref{rekurzija_vm}) and (\ref{rekurzija_wn}), respectively satisfy the following congruences
\begin{equation}\label{v2mc}
    v_{2m} \equiv z_0 \pmod{c}, \ \ v_{2m+1} \equiv v_1 \pmod{c},
\end{equation}
\begin{equation}\label{w2nc}
    w_{2n} \equiv z_1 \pmod{c}, \ \ w_{2n+1} \equiv w_1 \pmod{c}.
\end{equation}
\end{lemma}

The next lemma is proven in \cite[Lemma 6]{ff} for polynomials in $\mathbb{Z}[X]$, but the same arguments are true for $D(4)$-quadruples in $\mathbb{Z}[i][X]$.
\begin{lemma}\label{lemma 5.1}
Let the sequences $(v_m)_{m\geq0}$ and $(v_n)_{n\geq0}$ be given by (\ref{rekurzija_vm}) and (\ref{rekurzija_wn}). Then,
\begin{align*}
    v_{2m}&\equiv z_0+\frac{1}{2}c(az_0m^2+sx_0m)\ (\bmod \ c^2),\\
    2v_{2m+1}&\equiv sz_0+c\left(\frac{1}{2}asz_0m(m+1)+x_0(2m+1)\right)\ (\bmod \ c^2),
    \\
    w_{2n}&\equiv z_1+\frac{1}{2}c(bz_1n^2+ty_1n)\ (\bmod \ c^2),\\
    2w_{2n+1}&\equiv tz_1+c\left(\frac{1}{2}btz_1n(n+1)+y_1(2n+1)\right)\ (\bmod \ c^2).
\end{align*}

\end{lemma}

It is also easy to check that the next lemma holds.

\begin{lemma}\label{lemma 5.2}
  Let $\{a,b,c\}$ be a $D(4)$-triple from $\mathbb{Z}[i][X]$ for which (\ref{jdbe_osnovne}) holds. Then
\begin{equation}\label{novo}
rst\equiv 2a+2b-2d_-\ (\bmod \ c).
\end{equation}
\end{lemma}

Now, we introduce Lemma \ref{alpha0} as a useful tool for later results, i.e. for the proof of Proposition \ref{proposition_1_jurasic} and Theorem \ref{tm_glavni}.

\begin{lemma}\label{alpha0}
If $\alpha =0$, $x_0$ is a constant and $z_0$ is not a constant, then $x_0^2 = a^2+4$ and $z_0 = \pm s$. Also, $x_0=0$ and $a=\pm 2i.$
\end{lemma}
\begin{proof}
Let $\alpha =0$ and $x_0$ be a constant. 
 Define
$$e=\frac{x_0^2-4}{a}\in\mathbb{C}.$$
It is easy to see from (\ref{jdba_pellova_prva}) that $ce+4=z_0^2$. Since $z_0$ and $c$ are not constants,  $\{a, c, e\}$ is a polynomial $D(4)$-triple over $\mathbb{C}$, so by Lemma \ref{lema_samo_jedan_konstantan} we must have $a=e$. Now, $x_0^2 = a^2+4$ and inserting that in (\ref{jdba_pellova_prva}) implies $z_0=\pm s$. Also, $(x_0-a)(x_0+a)=4$ imply $a=\pm 2i$ and $x_0=0$ are the only possibilities.
\end{proof}

\section{Gap principle for degrees}\label{gap_degrees} 

In this section, we obtain all the possible inequalities for the degrees of 
polynomials in a $D(4)$-triple $\{a, b, c\}$, i.~e.~we deal with the inequalities that degrees $\alpha, \beta, \gamma$, satisfy. 

\begin{lemma}
Let $\{a, b, c\}$ be a $D(4)$-triple in $\mathbb{Z}[i][X]$ such that (\ref{jdbe_osnovne}) holds. For $d_{\pm}$ defined by (\ref{jdba_d_pm}),
where $u_{\pm}, v_{\pm}$, and $w_{\pm}$ are introduced in (\ref{uvw}),
we get
\begin{equation}\label{izrazc}
    c = a + b + d_{\pm} + \frac{1}{2}abd_{\pm}  \mp \frac{1}{2}ru_{\pm}v_{\pm}
\end{equation}
 \begin{equation}\label{rw}
     d_{\pm} = a+b-c+rw_{\pm},
 \end{equation}
  \begin{equation}\label{sw}
     d_{\pm} = a-b+c+sv_{\pm},
 \end{equation}
  \begin{equation}\label{tu}
     d_{\pm} = -a+b+c+tu_{\pm}.
 \end{equation}
\end{lemma}

\begin{proof}
Analogously as in \cite[Lemma 1]{dft} we get 
(\ref{izrazc}) and (\ref{rw}) and, by (\ref{rw}), the expressions (\ref{sw}) and (\ref{tu}) are obtained. 
\end{proof}
Similarly as in \cite[Lemma 1]{dfl} (and \cite[Lemma 2]{djint}) we can prove a gap in degrees of polynomials in the polynomial $D(4)$-triple $\{a,b,c\}$, where $a, b, c \in\mathbb{Z}[i][X]$. Also, if we observe an extension of a $D(4)$-pair $\{a,b\}$ to a $D(4)$-triple $\{a,b,c\}$, we can describe in more details the element $c$ that has the lowest degree.

\begin{lemma}\label{degsc}
Let $\{a, b, c\}$ be a polynomial Diophantine $D(4)$-triple such that $\alpha, \beta, \gamma$ are degrees of the polynomials $a, b, c$, respectively, and $\alpha \leq \beta \leq \gamma$. Then $$c = a + b \pm 2r \ \ \text{or} \ \ \gamma \geq \alpha + \beta.$$
\end{lemma}

\begin{proof}
For (\ref{jdbe_osnovne}) we consider the polynomials
 of the form
 $$d_{1, 2}= a + b+ c + \frac{1}{2}(abc\pm rst).$$
Because $$d_1\cdot d_2 = a^2 +b^2 + c^2 - 2ab-2ac-2bc-16, $$ we notice that $\deg(d_1)+\deg(d_2)\leq 2\gamma$ and because at most one element of the Diophantine $D(4)$-triple $\{a, b, c\}$ is a constant polynomial, we conclude that $\deg(d_1) \neq \deg(d_2)$. We denote by $d_{-}$ the polynomial with the lower degree among $d_1, d_2$. In this case, $\deg(d_-)<\gamma.$ 

We set $$c_{\pm} = a+b+d_-+\frac{1}{2}(abd_-+ru_{\pm}v_{\pm}),$$ for $u_{\pm}, v_{\pm}$ that are already introduced in (\ref{uvw}).

After some elementary calculations, we can notice that
$$ac_+ + 4 = \left(\frac{1}{2}(av_{\pm}+ru_{\pm})\right)^2 = \left(\frac{1}{4}(abs\pm art + art \pm abs \pm 4s)\right)^2,$$
$$ac_- + 4 = \left(\frac{1}{2}(av_{\pm}-ru_{\pm})\right)^2 = \left(\frac{1}{4}(abs\pm art - art \mp abs \mp 4s)\right)^2.$$
There exists an index $i \in\{+, -\}$ such that $ac_i +4 = s^2 = ac+4$, so $c=c_i$. Let $c' = c_j$ such that $j\neq i$, $ j \in \{+, -\}$. For this set up we have $cc' = c_+ c_- = a^2 + b^2 + d_-^2 -2ab-2bd-2bd-16$, so $\deg(c) + \deg(c') \leq 2\gamma$. Hence, $\deg(c)\geq \deg(c')$. We have two possibilities, if $d_- = 0$, then $u_{\pm} = \pm 2$, $v_{\pm} = \pm 2$, and $c_{\pm} = a + b \pm 2r$,
and if $d_- \neq 0$, then $\gamma \geq \deg(abd_-) \geq \alpha + \beta$.
\hfill\qedhere
\end{proof}
Now we deal with $D(4)$-quadruples $\{a, b, c, d\}$ for which $d_-$ is a possible $d$ and obtain all possibilities for $d_-$. Additionally, we use obtained relations to draw some conclusions about the degrees of the polynomials $a, b, c$.
\begin{remark}\label{degdplusminus}
By (\ref{jdba_d_pm}), it is easy to conclude that 
\begin{equation}\label{degdplus}
\deg(d_+) = \alpha + \beta + \gamma > \gamma.
\end{equation}
The next lemma shows that for $d_{-}\neq0$ it holds 
\begin{equation}\label{degdminus}
    0\leq \deg(d_{-})\leq \gamma - \alpha - \beta,
\end{equation}
or, specifically, $\gamma \geq \alpha + \beta.$
\end{remark}

The proof of the next lemma is analogous to \cite[Lemma 2]{dujjur}.

\begin{lemma}\label{degdminuslemma}
Let $\{a, b, c\}$ be a $D(4)$-triple in $\mathbb{Z}[i][X]$ and let $d_{-}$ be defined by (\ref{jdba_d_pm}). Then, $d_{-} = 0$ or $\deg(d_{-}) = \gamma - \alpha - \beta<\gamma$.
\end{lemma}

\begin{remark}\label{Rem_dminus}
From previous Lemma \ref{degdminuslemma}, in case $\beta = \gamma$ we get $d_{-}=0$ or $\deg(d_{-})=0$, so $d_{-}=a=\pm 2i.$\par
If $d_{-}=0$, then $c=a+b\pm 2r$ and, by (\ref{jdba_c_regularni1}), $s=\pm(a\pm r)$ and $t=\pm(b\pm r)$. More precisely, if $c=a+b+2r$ then $s=\pm(a+r)$ and $t=\pm(b+r)$ and if $c=a+b-2r$ then $s=\pm(a-r)$ and $t=\pm(b-r)$.\par
If $d_{-}=a=\pm 2i$, then $c=-b+2a=-b\pm 4i$ and, by (\ref{jdba_c_regularni1}), $s=\pm ir$ and $t=\pm i(b-a)$.
\end{remark}

\begin{remark}
Using Remark \ref{Rem_dminus} we can easily determine some examples of polynomial $D(4)$-triples over Gaussian integers.

Let $d_{-} = a = \pm2i$. From $ab + 4 = r^2$ and  $bd_{-}+4 = v^2$ we get $$\pm 2bi + 4 = r^2 = v^2.$$  
So, $v = \pm r$, or $\pm r = \frac{1}{2}(bs \pm rt).$ 
For constructing our example, we will observe the case when $a=2i$ and $r=\frac{1}{2}(bs + rt)$.
Hence,
\begin{equation}\label{bs}
    r( 2- t) = bs.
\end{equation}
Moreover, we will assume $s|r$. Then there exists $p\in\mathbb{Z}[i][X]$ such that $r=ps$. Now (\ref{bs}) implies
\begin{equation}\label{pt2p}
    b =  2p - pt.
\end{equation} 
From $c=-b+2a= -b + 4i$, we get $C = -B$, where $B, C$ are leading coefficients of $b, c\in\mathbb{Z}[i][X]$, respectively. By (\ref{pt2p}) and (\ref{jdbe_osnovne}), $B =- p \sqrt{BC}$. So, we can choose $p = i$, so we observe a $D(4)$-triple
$$\{a,b,c\}=\{2i,-ti+2i,ti+2i\}.$$
It is easy to check that $r^2=2t$ and $s^2=-2t$. So we can choose any polynomial $t\in\mathbb{Z}[i][X]$ such that $2t$ is a perfect square. For example, we can take $t=2(X+1)^2$ and get
$$\{a,b,c\}=\{2i,-2X^2i-4Xi,2X^2i+4Xi+4i\}.$$
\end{remark}


If $d=d_-$ then we can easily prove what values $m$ and $n$ can have. 

\begin{lemma}\label{gap0}
Let $\{a, b, c\}$ be a $D(4)$-triple in $\mathbb{Z}[i][X]$. Let $v_m=w_n$ and let $d = \dfrac{v_m^2-4}{c}$.
If $d = d_{-}$, then $m, n \in \{0, 1\}$.
\end{lemma}
\begin{proof}
If $d_{-} = 0$, from $cd_{-}+4 = w_{-}^2$ we simply obtain $w_{-} = \pm 2$, so $\deg(w_{-}) = 0$. If $d_{-} \neq 0$, by Lemma \ref{degdminuslemma} and $cd_{-}+4 = w_{-}^2$ we obtain $$\deg(w_{-}) = \gamma - \frac{\alpha+\beta}{2} < \gamma.$$ By (\ref{degvm}) and (\ref{degwn}) we have $\deg(v_m), \ \deg(w_n) \geq \gamma$ for $m, n\geq 2$, respectively. So, $d_{-}$ must arise from $v_m = w_n $ for $m, n \in \{0, 1\}$.
\end{proof}

The next lemma
considers all possibilities for $d_-$, so we get all the possible relations between degrees $\alpha, \beta$ and $\gamma$. A similar gap principle is very well known in the classical case and in the polynomial variants of the problem of Diophantus (see \cite{fj_19}, \cite{df}).

\begin{lemma}\label{degreesd}
Let $\{a, b, c\}$ be a $D(4)$-triple in $\mathbb{Z}[i][X]$ for which (\ref{jdbe_osnovne}) holds. 
Let $A, B, C$ denote the leading coefficient of $a, b, c,$ respectively.
Then:
\begin{enumerate}
\item If $d_{-} = 0$, then $z_0 = z_1 = \pm 2$. In this case, $c = a + b \pm 2r$ and $\beta = \gamma$. Additionally, if $\alpha <\beta=\gamma$, then $B = C$. If $\alpha = \beta=\gamma$, then $C = A + B \pm 2\sqrt{AB}$.

\item For $\deg(d_{-})=0$, the following options can occur
\begin{enumerate}

   \item If $d_{-}=a =\pm 2i$, then $z_0= z_1 =\pm s$, for $\alpha =0, \ \beta = \gamma$ and $c = -b \pm 4i.$ 
   
\item If $d_{-}\in\mathbb{Z}[i]\backslash\{0, a\}$, then $z_0 = z_1 = \pm \frac{1}{2}(cr\pm st),$ for $\alpha>0, \gamma = \alpha + \beta.$
 
\end{enumerate}

\item For $\deg(d_{-})>0$, the following options can occur

\begin{enumerate}
    \item $z_0 = z_1 = \pm \frac{1}{2} (cr \pm st)$ with $\alpha >0$, $\deg(d_{-})\leq \alpha$ and $\alpha + \beta <\gamma \leq 2\alpha + \beta$,
    \item $(z_0, z_1) = (\pm \frac{1}{2} (cr \pm st), \pm s)$ with $\alpha \leq \deg(d_{-}) \leq \beta$, $\alpha \geq 0$ and $2\alpha + \beta \leq \gamma \leq \alpha + 2\beta$,
    \item $(z_0, z_1) = (\pm t, \pm \frac{1}{2} (cr \pm st))$ with $\deg(d_{-})=\alpha$, $\alpha=\beta$ and $\gamma = 3\alpha$,
    \item $(z_0, z_1) = (\pm t, \pm s)$ with $\beta \leq \deg(d_{-}) <\gamma$, $\alpha \geq0$ and $\gamma \geq \alpha+2\beta.$
\end{enumerate}

\end{enumerate}
\end{lemma}
\begin{proof}
 \emph{1.} For $d_{-} = 0$ from (\ref{izrazc}) we get $c_{} = a + b \mp \frac{1}{2}ruv$ and because in this case we have $u = v = \pm 2$, obviously we obtain $c = c_{\pm} = a + b \mp 2r$. Hence, we get $\gamma = \beta.$ In case $\alpha < \beta$, we get $C = B$ and in case $\alpha = \beta$, we get $C = A + B \mp 2\sqrt{AB}$.

From $cd_{-} + 4 = w^2$ we get $w = \pm 2$. Since Lemma \ref{lema_stupnjevi} implies $\deg(w_n)\geq\gamma/2$ for $n\geq 1$, we have $w=w_0 =z_1= \pm 2$ and, similarly, $v=v_0=z_0=\pm 2$.
\smallskip 

    \emph{2.(a)}
    First, we assume that $\alpha=0$ and, as we earlier saw, because we cannot have two different constants in a $D(4)$-quadruple in $\mathbb{Z}[i][X]$, we conclude $d_{-}=a\in\mathbb{Z}[i]$. By (\ref{uvw}) we get $$(a+u)(a-u)=-4, \ \ a, u \in\mathbb{Z}[i].$$ This is possible only for $a=\pm 2i, \ u=0$. Furthermore, we easily conclude
    \begin{equation}\label{lemac}
        c =-b+2a= -b \pm 4i
    \end{equation}
    by (\ref{tu}). So, in this case, we have $\beta = \gamma$.
 The equation $cd_{-} + 4 = w^2$ becomes $\pm 2ci + 4 = w^2$ and, because $\pm 2ci + 4 = s^2$, we get $w = \pm s$. Using  Lemma \ref{gap0} the equation $v_m = w_n = \pm s$ for $m, n\in\{ 0, 1\}$ is obtained.
Now we deal with each case separately.

If we have $(m, n)=(0, 0)$, by (\ref{rekurzija_vm}) and (\ref{rekurzija_wn}), we get $z_0 = z_1 = \pm s$.\\
For $(m, n) = (0, 1)$, by (\ref{rekurzija_vm}) and (\ref{rekurzija_wn}), we get 
\begin{equation}\label{z0z1}
z_0 = \frac{1}{2}(tz_1 + cy_1) = \pm s.
\end{equation}
By (\ref{uvw}) and (\ref{z0z1}) we easily get $\pm \frac{1}{2}(cr \pm st) = \frac{1}{2}(tz_1+cy_1)$, or equivalently,
\begin{equation}\label{lemma2a}
 c(\pm r - y_1) = t(z_1 \mp s).   
\end{equation}
Let $g=\gcd(c,t)$. Because $bc+4=t^2$, we conclude $g\in\mathbb{Z}[i]$. Then $c=gc_1$, $t=gt_1$ for some polynomials $c_1,t_1\in\mathbb{Z}[i][X]$ and $\deg(c_1)=\gamma=\deg(t_1)$. Since $\gcd(c_1,t_1)=1$, we have  $$t_1\mid (\pm r - y_1).$$ From $ab+4=r^2$ we know $\deg (r)=\frac{\gamma}{2}$ and from (\ref{ineq20}) we have $\deg (y_1)\leq \frac{\gamma}{2}$, so, the only possibility is $y_1 = \pm r$ and $z_1 = \pm s$. \\
If $(m, n)=(1, 0)$, we get the equation $\frac{1}{2}(sz_0+cx_0) = z_1 = \pm s,$ or equivalently, $cx_0 = s(\pm 2 - z_0).$
Similarly as before, let $g=\gcd(c,s)$. From $ac+4=s^2$ we know that $g\in\mathbb{Z}[i]$ and there exist $c_1,s_1\in\mathbb{Z}[i][X]$, $\deg(c_1)=\gamma$, $c=gc_1$, $s=gs_1$. Then
$$c_1\mid (\pm2-z_0).$$
From (\ref{deg1}) we get $\deg (z_0)<\gamma$, so we must have $z_0=\pm 2$. But then $cx_0=0$, which implies $x_0=0$ and that is a contradiction to (\ref{jdba_pellova_prva}).\\
Finally, if we set $(m, n) = (1, 1)$, we obtain $\frac{1}{2}(sz_0+cx_0) = \frac{1}{2}(tz_1+cy_1) = \pm s$. Similarly as in case $(m, n) = (1, 0)$, this is not possible.
\smallskip 

\emph{2.(b)}\ Let $d_{-}\in\mathbb{Z}[i]\backslash \{0, a\}$. Because $d_{-}$ is a constant polynomial and we cannot have two or more constants in a $D(4)$-quadruple, we conclude $\alpha>0$.  By Lemma \ref{degdminuslemma}, we know $\gamma = \alpha + \beta$. Using (\ref{uvw}), we get $w = \pm \frac{1}{2}(cr\pm st)\neq s,$ so from Lemma \ref{gap0} we conclude $v_m = w_n = \pm \frac{1}{2}(cr\pm st),$ for some $m, n\in\{0, 1\}$. 

Setting $(m, n)=(0, 0)$, we obtain $z_0 = z_1 = \pm \frac{1}{2}(cr\pm st)$.\\
In the case $(m, n) = (0, 1)$, we get $z_0 = \frac{1}{2}(tz_1 + cy_1) = \pm \frac{1}{2}(cr\pm st)$, which again implies (\ref{lemma2a}), or $y_1 = \pm r$ and $z_1 = \pm s$. From $\gamma = \alpha + \beta$ and (\ref{ineq20}) we easily obtain a contradiction.\\
For $(m, n) = (1, 0)$ we deal with $\frac{1}{2}(sz_0+cx_0) = z_1 = \pm \frac{1}{2}(cr\pm st)$ and get $x_0 = \pm r, \ z_0 = \pm t$, which leads again to contradiction using (\ref{deg1}). \\
Finally, by analogous reasoning as in the cases $(m, n)=(0, 1)$ and $(m, n)=(1, 0)$, we obtain a contradiction in the last case for $(m, n) = (1, 1)$.

\smallskip
 \emph{3.} For $\deg(d_{-})>0$ by Lemma \ref{degdminuslemma} we have $\gamma > \alpha + \beta$. Also, by Lemma \ref{gap0}, $d_{-}$ occurs from $v_m = w_n = \pm\frac{1}{2}(cr \pm st)$ for some $m, n\in \{0, 1\}$. From $\deg (d_-)=\gamma-\alpha-\beta$ and $cd_-+4=w^2$, we get
\begin{equation}\label{degw}
    \deg(w) = \gamma -\frac{\alpha+\beta}{2} < \gamma.
\end{equation}
For start, we consider $(m, n) = (0, 0)$ and, using (\ref{rekurzija_vm}) and (\ref{rekurzija_wn}), we get that $z_0 = z_1 = \pm\frac{1}{2}(cr\pm st)$. By (\ref{degw}) and (\ref{ineq20}), we get $\gamma \leq 2\alpha + \beta$, which by Lemma \ref{degdminuslemma} implies $\deg(d_{-})\leq \alpha$. Thus, $\alpha > 0.$

In the case $(m, n)=(0, 1)$, we again observe the equalities (\ref{z0z1}) and (\ref{lemma2a}). As before we conclude $y_1 = \pm r$ and $z_1 = \pm s.$ Analogously, by (\ref{degw}) and (\ref{deg1}), we get $\gamma \leq \alpha + 2\beta$ which by Lemma \ref{degdminuslemma} leads us to conclusion that $\deg(d_{-})\leq \beta.$ Since $\deg (z_1)=\deg (s)=\frac{\alpha+\gamma}{2}$, inequality (\ref{ineq20}) leads us to $\gamma \geq 2\alpha + \beta$, so we easily get $\deg(d_{-})\geq\alpha.$

The next case that we deal with is $(m, n) = (1, 0)$, where similarly as in the previous case we get $x_0 = \pm r, \ z_0 = \pm t$ for $z_1 = \pm\frac{1}{2}(cr\pm st)$. We get $\alpha + 2\beta \leq \gamma \leq 2\alpha + \beta$ which implies $\alpha = \beta$ and $\gamma = 3\alpha.$ 

Finally, for $(m, n)=(1, 1)$, by following similar arguments as before, we obtain $z_0 = \pm t$ and $z_1 = \pm s$ and consequently, by (\ref{deg1}) and (\ref{ineq20}), we get $\beta \leq \deg(d_{-}) < \gamma$ and $\alpha \geq0$ with $\gamma \geq \alpha+2\beta.$
\end{proof}

\section{Precise definition of the initial terms }

Now it is time to determine all the possible initial terms of the recurring sequences $(v_m)_{m\geq0}$ and $(w_n)_{n\geq0}$ for the extension of the $D(4)$-triple $\{a, b, c\}$ in $\mathbb{Z}[i][X]$. 

\begin{lemma}
\label{initial2}
\begin{enumerate}\mbox{}
    \item If the equation $v_{2m}=w_{2n}$ has a solution, then $z_0 = z_1$.

    \item  If the equation $v_{2m+1}=w_{2n}$ has a solution, then either $(z_0, z_1) = (\pm2, \pm s)$ or $(z_0, z_1) = (\pm s, \pm 2)$ or $z_1 = \frac{1}{2}(sz_0\pm cx_0)$ (where $x_0$ is not a constant polynomial.)
    \item If the equation $v_{2m} = w_{2n+1}$ has a solution, then either $(z_0, z_1) = (\pm t, \pm 2)$ or $(z_0, z_1) = (\pm s, \pm 2)$ or 
    $z_0 = \frac{1}{2}(tz_1\pm cy_1),$ (where $y_1$ is not a constant polynomial.)

    \item If the equation $v_{2m+1}=w_{2n+1}$ has a solution, then either $(z_0, z_1) = (\pm 2, \frac{1}{2}(\pm cr \pm st))$, or $(z_0, z_1) = (\frac{1}{2}(\pm cr \pm st), \pm 2)$, or $\frac{1}{2}(sz_0+cx_0) = \frac{1}{2}(tz_1\pm cy_1)$
or   $\frac{1}{2}(sz_0-cx_0) = \frac{1}{2}(tz_1\pm cy_1)$.
\end{enumerate}
\end{lemma}

\begin{proof}

 \emph{1.} For $v_{2m} = w_{2n}$ we use the congruences from Lemma \ref{initial1} and get $z_0 \equiv z_1 \pmod{c}$. According to  (\ref{deg1}), we conclude $\deg(z_0)<\gamma,$ $\deg(z_1) < \gamma$, and consequently $z_0 = z_1$.
    
 \emph{2.} We have $v_{2m+1} = w_{2n}.$ According to congruences from Lemma \ref{initial1} we conclude $v_1 \equiv z_1 \pmod{c}$, or more precisely $sz_0 + cx_0 \equiv 2z_1 \pmod{c}$, i.e. $sz_0 \equiv 2z_1 \pmod{c}$.\\
    First, we observe the case $z_0=\pm 2$ and $x_0=\pm 2$. If $\alpha<\gamma$, for $z_0 = 2,$ we get $z_1 = s$ and for $z_0=-2$ we get $z_1=-s$. If $\alpha=\beta=\gamma$, from Lemma \ref{degsc} we know $c=a+b\pm 2r$, and by observing degrees of $v_{2m+1}$ and $w_{2n}$ we get a contradiction.\\
    For $z_0$ which is not a constant polynomial, we know from Lemma \ref{kongruencije_rj} that $\deg(z_0) \geq \frac{\gamma}{2},\  \deg(x_0) \geq \frac{\alpha}{2}.$ 
    We first assume that $x_0$ is a constant polynomial, meaning $\alpha=0$.  Lemma \ref{alpha0} implies $z_0=\pm s$. Since $s^2\equiv 4\ (\bmod \ c)$, from $sz_0 \equiv 2z_1 \pmod{c}$ we have $z_1= 2$ if $z_0=s$ and $z_1=-2$ if $z_1=-s$.\\
    If $x_0$ is not a constant, then from
\begin{equation}\label{lemma41-2}
(cx_0+sz_0)(cx_0-sz_0) = 4c^2 - 4ac - 4z_0^2,
\end{equation}
    we conclude that one of two polynomials, $(cx_0+sz_0)$ or $(cx_0-sz_0)$ has degree less that $\gamma$, and it is congruent to $2z_1 \pmod{c}$. Hence, equality $z_1=\frac{1}{2}(cx_0\pm sz_0)$ holds.

 \emph{3.} 
 According to Lemma \ref{initial1} we have $z_0 \equiv \frac{1}{2}(tz_1 + cy_1) \pmod{c}$. \\
    For $z_1 = \pm 2$, we first observe a case when $\beta<\gamma$ and get  $2z_0 =tz_1$, implying $(z_0,z_1)=\pm(t,2)$.
    If $\beta=\gamma$, then from Lemma \ref{degdminuslemma} we have $\deg (d_-)=0$ and from Remark \ref{Rem_dminus} we know $d_-=a=\pm 2i$ or $c=a+b\pm 2r$.
    In the first case, we get $z_0=\pm 2$. Then $y_1=\pm 2$ and $x_0=\pm 2$, so $v_1=\pm s \pm c$ and $w_1=\pm t\pm c$. Since $c=-b+2a$ and $t=\pm i(b-a)$, we see that $\deg (v_1)=\deg (w_1)=\gamma$. By using Lemma \ref{lema_stupnjevi} we get a contradiction. In second case, when $c=a+b\pm 2r$, we get $z_0=\pm s$.\\
    If $z_1$ is not a constant polynomial, then we know that $\deg(z_1) \geq \frac{\gamma}{2}$ and $\deg(y_1)\geq \frac{\beta}{2}$.  We have 
    \begin{equation}\label{initial200}
    \frac{1}{4}(tz_1+cy_1)(tz_1-cy_1) = bc - c^2 +z_1^2
    \end{equation}
    and one of the polynomials $\frac{1}{2}(tz_1+cy_1)$ or $\frac{1}{2}(tz_1-cy_1)$ has degree less than $\gamma$ and they are both congruent to $z_0 \pmod {c}$. So, one of these polynomials is $z_0$.

\emph{4.} Finally, in the last case we have $v_{2m+1} = w_{2n+1}$, and analogously as in previous cases, we conclude $sz_0 \equiv tz_1 \pmod{c}$. If $x_0, y_1$ are not constants, then again we conclude that one of the polynomials $\frac{1}{2}(cx_0+sz_0), \frac{1}{2}(cx_0-sz_0)$ and one of the polynomials $\frac{1}{2}(cy_1+tz_1), \frac{1}{2}(cy_1-tz_1)$ have degrees less than $\gamma$ and are congruent to each other modulo $c$, so they have to be equal.

    If $z_0 = \pm 2$, we get $4z_1 \equiv \pm2  st \pmod{c}$, which can be expressed as $4z_1 \equiv \pm
    2 (st \pm cr) \pmod{c}$. In the equality $$4(\pm st + cr)(\pm st -cr) = 16ac+16bc+64-16c^2, $$ one of the polynomials within parentheses has degree less than $\gamma$, the other one obviously has degree $\gamma + \frac{\alpha+\beta}{2}$. So, $\frac{1}{2}(\pm st \pm cr) = z_1$ and in this case $\deg (z_1) \leq \gamma - \frac{\alpha+\beta}{2}$. Recall that we have $cd_-+4 = z_1^2$, so $\deg (z_1)=\frac{\gamma}{2}+\frac{\deg (d_-)}{2}$.
    \\ If $\deg(z_1) < \gamma -\frac{\alpha+\beta}{2}$, Lemma \ref{degdminuslemma} implies $d_-=0$, $\beta=\gamma$ and $z_1 = \pm 2$.
     \\ If $\deg(z_1) = \gamma -\frac{\alpha+\beta}{2}$, because it has to be $\deg(z_1)\leq \frac{3\gamma-\beta}{4}$, it follows $\gamma \leq 2\alpha + \beta$. As a special case, let's observe the case $\beta = \gamma$. If $\alpha > 0$, we get $\deg(z_1)<\frac{\gamma}{2}$. So, in this case, we have $d_-=0$ and $z_1 = \pm 2$. If $\alpha = 0$, then $\deg(z_1) = \frac{\gamma}{2}$. So, $\deg(d_-)=0$ and $d_-=a$ while $z_1 = \pm s$.

    Now we deal with the case when $z_0\neq \pm 2$ and $x_0$ is a constant. As above, $z_0 = \pm s$, so $\pm \frac{1}{2} s^2 \equiv \frac{1}{2}tz_1 \pmod{c}$, or $\pm 2 \equiv \frac{1}{2}tz_1 \pmod{c}$. After multiplying the congruence by $t$, we obtain $\pm 2t \equiv \pm 2z_1 \pmod{c}$. If $\beta < \gamma$, we have $z_1 = \pm t$ and $y_1^2 = b^2 + 4$, a contradiction. If $\beta = \gamma$, we can write $\pm 4\equiv tz_1 \pm cy_1 \pmod{c}$. If $\deg(z_1)\geq \frac{\gamma}{2}$, then $\deg(y_1)\geq \frac{\beta}{2}$ and $y_1 $ is not a constant. As above, we obtain a contradiction.
    
   Consider now a general case when $\alpha \leq \beta \leq \gamma$ and $z_1 = \pm 2$. After multiplying the congruence $sz_0 \equiv tz_1 \pmod{c}$ by $s$, we get $4z_0 \equiv \pm 2st \pmod{c}$ and we can write $4z_0 \equiv \pm 2(cr\pm st) \pmod{c}$. Since $\deg(z_0) < \gamma - \frac{\alpha+\beta}{2}$, we conclude $z_0 = \pm\frac{1}{2}(cr \pm st)$. As in the case $z_1=\pm\frac{1}{2}(cr \pm st)$ we can observe separately the case $\beta=\gamma$. If $\alpha>0$, we have  $d_-=0$.  Since $ad_-+4=z_0^2$, then $z_0=\pm 2$. If $\alpha=0$ then $\deg(d_-)=0$ and $d_-=a$ imply $z_0=\pm s$.
 \hfill\qedhere
\end{proof}

By using Lemmas \ref{degdminuslemma} and \ref{degreesd} we get the next result.
\begin{lemma}\label{lemma 3.7}
  Let $\{a, b, c\}$ be a $D(4)$-triple in $\mathbb{Z}[i][X]$ with $\beta < \gamma = \alpha + 2\beta$. In this case, $(a, b, d_-, c)$ is one of the following:
  \begin{enumerate}[i)]
      \item $(a,b,a+b+ 2r,r(r+ a)(b+ r))$ and $s=\pm(a(b+r)+2)$, $t=\pm (b(a+r)+2)$,
      \item $(a,b,a+b-2r,r(r- a)(b- r))$ and $s=\pm(a(b-r)+2)$, $t=\pm(b(r-a)-2)$,
      \item $( \pm 2i, \ b, \ -b\pm 4i, \ \mp 2ib^2- 8b \pm 10i)$ and $s=\pm(\mp2b+4i)$, $t=\pm (bir\pm r)$.
  \end{enumerate}
\end{lemma}


We consider the intersection of the recurrent sequences $(\ref{rekurzija_vm})$ and $(\ref{rekurzija_wn})$, i.~e.~we are interested in the solution of the equation $v_m = w_n$.  The next proposition treats all possible $(m, n)$ such that the condition $\{0, 1, 2\} \cap \{m, n\} \neq \emptyset$ and Lemma \ref{degineq} hold, proving that the only possibilities obtained in these cases are $d = d_+$ or $\deg(d)<\gamma$. 

\begin{proposition}\label{proposition_1_jurasic}
Let $\{a, b, c\}$ be a polynomial $D(4)$-triple. We assume that $v_m = w_n$. If $\{0, 1, 2\} \cap \{m, n\} \neq \emptyset$, then either $\deg(d) < \gamma$ or $d = d_+$. For $(m, n) \in \{ (1, 1), (1, 2), (2, 1), (2, 2)\}$ we get $d = d_+.$
\end{proposition}
\begin{proof}
From Lemma \ref{degineq} and the condition $\{0, 1, 2\} \cap \{m, n\} \neq \emptyset$ we get 
\begin{align*}(m, n) \in \{ &(0, 0), (0, 1), (1, 0), (1, 1), (1, 2), (2, 1),\\
&(3, 1), (2, 2), (2, 3), (3, 2), (4, 2), (5, 2)\}.
\end{align*}

{To illustrate specific details and different techniques, we provide proof only for the cases $(m, n) = (2, 2)$ and $(m, n) = (2, 3)$. Other cases can be proved similarly. The proof follows ideas from \cite[Proposition 1]{dujjur}}.

Let  $(m, n) = (2, 2)$, i.e. $z = v_2 = w_2$. From Lemma \ref{initial2}, we know that $z_0 = z_1$. Hence, (\ref{rekurzija_vm}) and (\ref{rekurzija_wn}) together with (\ref{jdbe_osnovne}) imply
\begin{equation}\label{m2n21}
sx_0-ty_1 = (b-a)z_0.\end{equation}
From the initial system of Pellian equations (\ref{jdba_pellova_prva}) and (\ref{jdba_pellova_druga}) we have
\begin{equation}\label{m2n22}
(b-a)^2z_0^2 = 4(b-a)^2 + (b-a)(cy_1^2-cx_0^2),
\end{equation}
and after squaring (\ref{m2n21}), we get
\begin{equation}\label{m2n23}
(sx_0-ty_1)^2 = (tx_0-sy_1)^2+(b-a)(cy_1^2-cx_0^2).
\end{equation}
Combining (\ref{m2n21}), (\ref{m2n22}) and (\ref{m2n23}), it is obtained
\begin{equation}\label{m2n24}
tx_0-sy_1 = \pm 2 (b-a).
\end{equation}
From (\ref{jdba_pellova_prva}) and (\ref{jdba_pellova_druga})  we have
$s^2(bx_0^2+4(a-b)) = as^2y_1^2,$
and (\ref{m2n24}) leads to
$a(tx_0\mp 2(b-a))^2 = as^2y_1^2.$
Hence,
\begin{equation*}\label{m2n26}
(b-a)(4x_0^2\pm 4atx_0+a^2t^2) = (b-a)(ab+4)(ac+4),
\end{equation*}
which leads to $2x_0\pm at =  \pm rs$.
Since $x_0 = \pm \frac{1}{2}( rs \pm at)$, from (\ref{jdba_pellova_prva}) we get $z_0=\pm\frac{1}{2}(cr\pm st)$, and since $z_1=z_0$, we have  $y_1=\pm\frac{1}{2}(rt\pm bs)$ from (\ref{jdba_pellova_druga}). There are $8$ combinations to observe, but only two of them satisfy (\ref{m2n21}) and (\ref{m2n24}), namely $x_0=\frac{1}{2}( rs \pm at)$, $z_0=-\frac{1}{2}(cr\pm st)$, $y_1=\frac{1}{2}(rt\pm bs)$ and $x_0=-\frac{1}{2}( rs \pm at)$, $z_0=\frac{1}{2}(cr\pm st)$, $y_1=-\frac{1}{2}(rt\pm bs)$. From those we obtain  $v_2 =\pm \frac{1}{2}( st \pm cr)$. It is now easy to see $d=(v_2^2-4)/c=d_{\pm}$, i.e. $d=d_+$ or $\deg (d)<\gamma$.

Let $(m, n) = (2, 3)$. We have $z = v_2 = w_3$.
We observe all possible $(z_0,z_1)$ from Lemma \ref{initial2}. First, consider the case $z_1 = \pm 2$ and $z_0 = \pm t.$ From (\ref{jdba_pellova_prva}) and (\ref{jdba_pellova_druga}), we get $x_0 = \pm r$ and $y_1 = \pm 2$. From the proof of Lemma \ref{initial2} we know that in this case $\beta < \gamma$. Then $\deg (v_1)=\gamma\pm(\alpha+\beta)/2$ and $\deg (w_1)=\gamma$. From Lemma \ref{lema_stupnjevi}
 we have $\deg(w_3) = \beta + 2\gamma$ and $\deg(v_2)\in\{(2\alpha+\beta+3\gamma)/2,(3\gamma-\beta)/2\}$. Since $v_2=w_3$, we get obvious contradictions in both cases.
   
 Now, let's deal with the case when $z_1 = \pm 2$ and $z_0 = \pm s$. According to the proof of Lemma \ref{initial2},  we have $c=a+b\pm2r$ and $\beta=\gamma$. In this case, $y_1 = \pm 2$ and from (\ref{jdba_pellova_prva}) we have $x_0^2 = a^2 + 4$ which implies $\deg (x_0)=\alpha=0$ and $x_0=0$, $a=\pm 2i$. Since we also know that $t=\pm(b\pm r)$ it is easy to see from (\ref{rekurzija_wn}) and (\ref{rekurzija_vm}) that $\deg (w_1)\in\{\gamma,\gamma/2\}$ and $\deg (v_1)=\gamma$. From Lemma \ref{lema_stupnjevi} we have $\deg (w_3)\in \{3\gamma,5\gamma/2\}$ and $\deg (v_2)=3\gamma/2$. So $v_2=w_3$ cannot hold for $\gamma>0$.

   Now, let's observe the case $z_1 \neq \pm 2$. Notice that (\ref{rekurzija_vm}) and (\ref{rekurzija_wn}) and $v_2=w_3$ imply 
   \begin{equation}\label{v2w3}
z = z_0 + \frac{1}{2}c(az_0+sx_0) = \frac{1}{2}tz_1+\frac{1}{2}c(btz_1+3y_1)+\frac{1}{2}bc^2y_1.
    \end{equation}
   
   From Lemma \ref{initial2}, we have $z_0 = \frac{1}{2}(tz_1\pm cy_1)$, such that $\deg (z_0)<\gamma$. First, we assume 
    $z_0 = \frac{1}{2}(tz_1-cy_1).$
    Then (\ref{v2w3}) implies
    \begin{equation}\label{v2w3add1}
    az_0 + sx_0 = b(tz_1+cy_1) + 4y_1.
    \end{equation}
    Since one of the polynomials $\frac{1}{2}(tz_1\pm cy_1)$ has the degree $\gamma + \deg(y_1)$, this must be the polynomial $\frac{1}{2}(tz_1+cy_1)$. By (\ref{v2w3add1}) we get $\deg(az_0 + sx_0) = \beta + \gamma + \deg(y_1)$. But, on the other hand, from (\ref{deg1}) we have $$\deg(az_0 + sx_0) \leq \max\{\deg(az_0), \deg(sx_0)\}\leq \frac{3\alpha+3\gamma}{4},$$ which implies $\gamma \leq -\beta$, a contradiction.
    
    Now we assume 
    $z_0 = \frac{1}{2}(tz_1+cy_1).$ Notice that $w_1=z_0$.
From (\ref{v2w3}) we have
\begin{equation}\label{v2w3add21}
    az_0 + sx_0 = 2y_1 + 2bz_0.
\end{equation}
 Since $z_0 \neq \pm 2$, by Lemma \ref{kongruencije_rj}, it is obtained $\deg(z_0)\geq \frac{\gamma}{2}$ and $\deg(x_0)\geq \frac{\alpha}{2}$. If $x_0$ is a constant then $\alpha = 0$. By Lemma \ref{alpha0} we have $z_0=\pm s$ and $x_0^2=a^2+4$, so $\deg(z_0)= \frac{\gamma}{2}$ and $x_0=0$. Then (\ref{v2w3add21}) becomes $az_0=2y_1+2bz_0$ and after comparing degrees on both sides of that equality we have $\gamma/2=\beta+\gamma/2$, which is a contradiction with $\beta>0$.
So, $x_0$ cannot be a constant. Since 
\begin{equation}\label{v2w3add3}
(az_0+sx_0)(az_0-sx_0) = 4a^2 -4ac-4x_0^2,
\end{equation}
 we get that one of the polynomials $az_0\pm sx_0$ has a degree less or equal to $\frac{\alpha+\gamma}{2}+\deg(x_0)$ and the other one has degree  $\frac{\alpha+\gamma}{2}-\deg (x_0)$. 
But (\ref{v2w3add21}) implies 
   $ \deg(az_0+sx_0) = \beta + \deg(z_0)\geq \beta+\gamma/2,$
so we must have $\deg(az_0+sx_0) = \frac{\alpha+\gamma}{2} + \deg(x_0)$. From  $\beta +  \deg(z_0)=\frac{\alpha+\gamma}{2}+\deg(x_0)$ and 
$\deg (z_0)=\frac{\gamma-\alpha}{2}+\deg (x_0)$ 
we conclude $\alpha = \beta$. If we transform (\ref{v2w3add21}) into 
\begin{equation}\label{v2w3add5}
    sx_0-az_0 = 2y_1+2z_0(b-a),
\end{equation}
we get 
\begin{equation}\label{v2w3add6}
    \deg(2y_1+2z_0(b-a))\leq \frac{\alpha+\gamma}{2}-\deg(x_0).
\end{equation}
If $\deg(b-a)>0$, then (\ref{v2w3add6}) leads us to the conclusion $\deg(x_0)<\frac{\alpha}{2}$, a contradiction. Hence, $\deg(b-a)=0$ and by (\ref{v2w3add6})  we obtain $\deg(z_0) = \frac{\gamma}{2}$, so $\deg(x_0) = \frac{\alpha}{2}$.  Additionally, we can rearrange (\ref{v2w3add5}) into the equation
\begin{equation}\label{v2w3add7}
    sx_0-2y_1-bz_0 = z_0(b-a),
\end{equation}
where the degree of the right-hand side is $\frac{\gamma}{2}$. From $z_0 = \frac{1}{2}(tz_1+cy_1)$ and (\ref{jdbe_osnovne}), we get $sx_0-2y_1-bz_0 = sx_0-\frac{1}{2}t(bz_1+ty_1)$, so
\begin{equation}\label{gammapol}
    \deg(sx_0-\frac{1}{2}t(bz_1+ty_1)) = \frac{\gamma}{2}.
\end{equation}
  First, let's observe the case $0<\alpha = \beta < \gamma.$ From (\ref{jdbe_osnovne}) and (\ref{jdba_pellova_druga}), we conclude that one of the polynomials $\frac{1}{2}(bz_1\pm ty_1)$ has a degree $\beta+\deg(z_1)$
 and the other  $\gamma - \deg(z_1)$. None of these possibilities satisfies (\ref{gammapol}).\\
   It remains to observe the case $0< \alpha = \beta = \gamma$. In this case we have $\deg(x_0) = \deg(z_0) = \deg(z_1) = \deg(y_1) = \frac{\gamma}{2}$. According to Lemma \ref{degsc} we get $c=a+b\pm 2r$. We set $b-a=k$, for $k\in \mathbb{Z}[i]$. From (\ref{jdbe_osnovne}) we get
   \begin{equation}\label{jdba_a_b_k}a^2+ka+4 = r^2.\end{equation} 
   If we denote the leading coefficients of the polynomials $a$ and $r$ by $A_1$ and $R_1$, respectively, it follows $A_1 = \pm R_1$. Note that in the case $A_1=-R_1$ we have $s=\pm(a-r)$ and $t=\pm(b-r)$ and if $A_1=R_1$ then $s=\pm(a+r)$ and $t=\pm(b+r)$. 
   Equality (\ref{jdba_a_b_k}) transforms into $(a+ 2)^2 - 4a + ka = r^2$.
   We conclude $$a(k- 4) = (r-a- 2)(r+a+ 2).$$
   One of the polynomials $r-a- 2$ and $r+a+ 2$ is a constant polynomial. Let's first assume that $r+a+ 2$ is a constant polynomial, i.e. $A_1=-R_1$. If we observe these polynomials as polynomials in $\mathbb{C}[X]$, we have $a\cdot p=r-a- 2$ where $p=(k-4)/(r+a+ 2)$ is a constant polynomial in $\mathbb{C}[X]$. Hence, $a(p+1)=r- 2$. After comparing leading coefficients, we have $A_1(p+1)=R_1=-A_1$, i.e. $p+1=-1$. So, 
$$k- 4=-2(r+a+ 2)=-2(r+a)- 4,$$
which implies 
$k=-2(r+a).$
Inserting it into previous equations leads to $r=-a\pm 2$, $b=a\mp 4$, $c=4a\pm 8$, $s=\pm(2a\mp 2)$ and $t=\pm(2a\mp 6)$. Notice that $s=\pm(t\mp 4)$. 

If we observe the case when $r-a- 2$ is a constant polynomial, we get $k=2(r-a)$, $r=a\pm 2$, $b=a\pm 4$ and $c=4a\pm 8$. 
Also $s = \pm( 2a \pm 2)$ and $t = \pm (2a \pm 6)$, so $s=\pm(t\pm 4)$.
  \\    After inserting everything calculated into (\ref{gammapol}), we get
    $$\deg(t(\pm x_0-\frac{1}{2}(bz_1+ty_1))\pm 4x_0) = \frac{\gamma}{2}.$$
    This implies $x_0 =\pm \frac{1}{2}(bz_1+ty_1)$ and 
    $\pm 4x_0 = (b-a)z_0$. Since $b-a=\pm 4$, we also have $x_0 = \pm z_0$. 
    From (\ref{jdba_pellova_prva}) we easily get $ z_0 = \pm2$ which implies $\gamma = 0$, a contradiction. 
\end{proof}

We can introduce the following gap principle which will be used in the main theorem. The proof is analogous as in \cite[Lemma 5]{dujjur},  so it is omitted.

\begin{lemma}\label{gapdeg}
Let $\{a, b, c, d\}$ be a polynomial $D(4)$-quadruple and let $\alpha, \beta, \gamma, \delta$ be the degrees of the polynomials $a, b, c, d$, respectively. We assume $\alpha \leq \beta \leq \gamma \leq \delta$. Then either $$\delta \geq \frac{3\beta +5\gamma}{2}>\gamma \ \ \text{or} \ \ d = d_+.$$
\end{lemma}

We can assume that $\{a, b, c, d'\}$ is an irregular polynomial $D(4)$-quadruple where $\alpha, \beta, \gamma, \delta$ are the degrees of the polynomials $a, b, c, d'$ respectively, $\alpha \leq \beta \leq \gamma$, \ $\beta, \gamma > 0$, and $\deg(d')=\delta$ such that $\delta$ is minimal possible among all irregular polynomial $D(4)$-quadruples over Gaussian integers. We will prove that there is no such quadruple. 

 In Lemma \ref{gap0} we have already shown if $d=d_-$, then $m,n\in\{0,1\}$. Now we consider other cases of interest.

\begin{lemma}\label{gap1}
Let $\{a, b, c\}$ be a $D(4)$-triple in $\mathbb{Z}[i][X]$. Let $v_m=w_n$ and let $d = \dfrac{v_m^2-4}{c}$.
\begin{enumerate}[a)]
    \item If $d = d'$, then $v_m = w_n = \pm z$, for $m, n\geq 3$,
    \item If $0 \in\{m, n\}$, then $d=d_{-}$ or $d=0\neq d_{-}$ or $d = \pm 2i \neq d_{-}$,
    \item If $(m, n) = (1, 1)$, then $d=d_{-}$ or $d=0\neq d_{-}$ or $d = \pm 2i \neq d_{-}$  or $d = d_{+}$ and $\gamma \geq \alpha + 2\beta$.
\end{enumerate}
\end{lemma}
\begin{proof}
    
    \emph{a)} If $d=d'$, by Proposition \ref{proposition_1_jurasic} we conclude $m, n\geq 3.$ 
    
    \emph{b)} If $0\in\{m, n\}$, then from Proposition \ref{proposition_1_jurasic} we have $\deg(d)<\gamma$. By Lemma \ref{degdminuslemma}, we can have $d = d_-$. Since $\deg(d_+)>\gamma$, we conclude $d\neq d_+$. Using Lemma \ref{gapdeg} 
    and the minimality assumption, the only possible irregular quadruples $\{a, b, c, d\}$ are those with $d=0$ or $d=\pm 2i.$
    
    \emph{c)} Again, by Proposition \ref{proposition_1_jurasic}, we either have $\deg(d)<\gamma$ or $d=d_+$. If $\deg(d)<\gamma$, for $\{a, b, c, d\}$ irregular $D(4)$-quadruple, by the minimality assumption, we obtain $d=0$ or $d=\pm 2i$ or $d=d_-$. If $d=d_+$, (\ref{uvw}) and (\ref{degdplus}) imply $\deg(w_+) = \gamma+\frac{\alpha+\beta}{2}$. Since $w_+=v_1$, inequality $\gamma \geq \alpha +2\beta$ is obtained from (\ref{degv1}).
\end{proof}

By using Lemmas \ref{gap0} and \ref{gap1} we can also emphasize some more details about the initial terms and degrees by listing each case from the proof of Lemma \ref{initial2} separately. The proof follows ideas from \cite[Lemma 4.3]{fj_19} and \cite[Lemma 4.3]{glavni}.

\begin{lemma}\label{initial3}
\mbox{}
\begin{enumerate}
    \item If $v_{2m}=w_{2n}$, then
    \begin{enumerate}[a)]
        \item $z_0 = z_1 = \pm 2$, or
        \item $z_0 = z_1 = \pm s$ and $\alpha =0$, or
        \item $z_0 = z_1 = \pm\frac{1}{2}(cr \pm st)$ and $\alpha >0, \ \alpha +\beta \leq \gamma \leq 2\alpha + \beta$.
    \end{enumerate}
    \item If $v_{2m+1}=w_{2n}$, then 
    \begin{enumerate}[a)]
    \item $(z_0,z_1)\in\{( 2, s),(-2,-s)\}$, $\gamma\geq 2\alpha+\beta$, or
    \item $(z_0,z_1)\in\{( s, 2),(-s,-2)\}$, $\alpha=0$, ($x_0 = 0, \ a=\pm 2i$), or
    \item $(z_0, z_1)=(\pm t, \pm \frac{1}{2}(cr\pm st))$ and $\alpha = \beta$ and $\gamma = 3\alpha$.
    \end{enumerate}
    \item If $v_{2m}=w_{2n+1}$, then
    \begin{enumerate}[a)]
        \item $(z_0, z_1) \in\{ ( t, 2),(-t,-2)\}$ and $\beta < \gamma$, or
        \item $(z_0, z_1) = (\pm s, \pm 2)$, $c=a+b\pm 2r$, $\alpha=0$ and $\beta=\gamma$
        or
        \item $(z_0, z_1) = (\pm \frac{1}{2}(cr \pm st), \pm s)$ and $\alpha \geq 0$ and $2\alpha + \beta \leq \gamma \leq \alpha + 2\beta$ (special case is $(z_0, z_1) = (\pm s, \pm s)$ and $\alpha = 0, \beta = \gamma$).
    \end{enumerate}
    \item  If $v_{2m+1} = w_{2n+1}$, then
    \begin{enumerate}[a)]
    \item $(z_0, z_1) = (\pm 2, \pm \frac{1}{2}(cr\pm st))$ and $\gamma\leq 2\alpha +\beta$,
    if $\beta=\gamma$ we have these special subcases
    \begin{enumerate}[i)]
        \item $(z_0, z_1) = (\pm 2, \pm 2)$, $z_0=z_1$, $\alpha >0$, $d_-=0$,
        \item $(z_0, z_1) = (\pm 2, \pm s)$, $\alpha=0$, or
    \end{enumerate}
    \item $(z_0, z_1) = ( \pm \frac{1}{2}(cr\pm st),\pm 2)$ and $\gamma\leq 2\alpha +\beta$,
    if $\beta=\gamma$ we have these special subcases
    \begin{enumerate}[i)]
        \item $(z_0, z_1) = (\pm 2, \pm 2)$, $z_0=z_1$, $\alpha >0$, $d_-=0$,
        \item $(z_0, z_1) = (\pm s, \pm 2)$, $\alpha=0$,  or
    \end{enumerate}
   \item  $(z_0,z_1)\in\{( t, s),(-t,-s)\}$ and $\gamma \geq \alpha + 2\beta$. 
   
    \end{enumerate}
\end{enumerate}
\end{lemma}

\begin{proof}
We present only a short proof of cases \emph{2.}\@ and \emph{4.} Other cases are proven similarly.

\emph{2.} We deal with the equation $v_{2m+1}=w_{2n}$. 
        \\ \emph{a)} 
    In the case $(z_0, z_1) = (\pm 2, \pm s)$, from (\ref{ineq20}) we have 
    $$\frac{\alpha+\gamma}{2}=\deg (z_1)\leq \frac{3\gamma-\beta}{4},$$
    implying $\gamma\geq 2\alpha + \beta.$
    \\ \emph{b)}  If $(z_0, z_1) = (\pm s, \pm 2)$, $x_0$ is a constant and $\alpha=0$. Also, from $x_0^2 = a^2 +4$, we conclude $x_0 = 0, \ a=\pm 2i$. 
    \\ \emph{c)}  The last possibility in this case is $z_1 = \frac{1}{2}(sz_0\pm cx_0)$ and $x_0$ is not a constant. Using (\ref{rekurzija_vm}) and (\ref{rekurzija_wn}) we get $\pm v_1 = w_0.$ 
    
    According to Lemma \ref{gap1}, we deal with one of the following possibilities $d=d_{-}$ or $d=0\neq d_{-}$ or $d = \pm 2i \neq d_{-}$. The case when $d = d_{-}$ is described in 3.c) of Lemma \ref{degreesd}. For $d=0\neq d_{-}$ we have $\frac{1}{2}(sz_0 \pm cx_0) = \pm 2$, which means that one of the polynomials $sz_0\pm cx_0$ is a constant. 
    If $\alpha = \gamma$, by Lemma \ref{gap1} we would get $\gamma = 0$, which cannot hold. So, we have $\alpha < \gamma$. Since 
    $$(sz_0 + cx_0)(sz_0-cx_0) = 4z_0^2+4ac-4c^2,$$ we conclude that $\deg((sz_0+cx_0)(sz_0-cx_0))=2\gamma$. One of these polynomials is a constant, and the other has a degree $2\gamma$ which is a contradiction because of (\ref{ineq20}). In the last case, when $d = \pm 2i \neq d_{-}$, we have $\frac{1}{2}(sz_0 \pm cx_0) = \pm s$. From $ac+4=s^2$, we see that $g=\gcd(s,c)$ is a constant polynomial, so $\frac{s}{g} | x_0$. Since $\deg\frac{s}{g}=\deg (s)=\frac{\alpha+\gamma}{2} $, this is not possible for $x_0\neq 0$, because it would be a contradiction with  (\ref{deg1}).
\medskip 

    \emph{4.} By Lemma \ref{initial2} we can have $(z_0, z_1) = (\pm2, \pm\frac{1}{2}(cr \pm st))$ or $(z_0, z_1) = (\pm\frac{1}{2}(cr \pm st), \pm2)$ or $sz_0 \pm cx_0 = tz_1 \pm cy_1$. 
\\ \emph{a)} In the case $(z_0, z_1) = (\pm2, \pm\frac{1}{2}(cr \pm st))$, as in Lemma \ref{degreesd} 3.a)
we have $\gamma\leq 2\alpha+\beta$. From the proof of Lemma \ref{initial2} we see special subcases.

\emph{b)} If $(z_0, z_1) = (\pm\frac{1}{2}(cr \pm st), \pm2)$ we make similar conclusions as in case a). 

\emph{c)} The proof of this case is the same as \cite[Lemma 4.3 (4)]{glavni}.
    \end{proof}

\section{Proof of Theorem \ref{tm_glavni}}\label{theproof}

Analogously as in \cite[Lemma 2.14.]{moj_cetvorke} and \cite[Lemma 4.4]{glavni} we can prove the next lemma.
\begin{lemma}\label{lemma 4.4}
Let $(v_{z_0,m})_{m\geq0}$ denote a sequence $(v_m)_{m\geq0}$ with an initial value $z_0$ and $(w_{z_1,n})_{n\geq0}$ denote a sequence $(w_n)_{n\geq0}$ with an initial value $z_1$. It holds that
$$v_{t,m}=v_{-\frac{1}{2}(cr-st),m+1}=-v_{\frac{1}{2}(cr-st),m+1}$$
and
$$v_{-t,m+1}=v_{\frac{1}{2}(cr-st),m}=-v_{\frac{1}{2}(-cr+st),m}$$  for each $m\geq 0$. Also, for $m\geq 0$
$$v_{t,m+1}=v_{\frac{1}{2}(cr+st),m} \quad\textit{ and }\quad v_{-t,m}=v_{\frac{1}{2}(-cr-st),m+1}.$$ 
Similarly, for each $n \geq 0$ it holds
$$w_{s,n}=w_{-\frac{1}{2}(cr-st),n+1}=-w_{\frac{1}{2}(cr-st),n+1}$$
and
$$w_{-s,n+1}=w_{\frac{1}{2}(cr-st),n}=-w_{\frac{1}{2}(-cr+st),n}.$$ 
Also, for $n\geq 0$
$$w_{s,n+1}=w_{\frac{1}{2}(cr+st),n} \quad\textit{ and }\quad w_{-s,n}=w_{\frac{1}{2}(-cr-st),n+1}.$$ 
\end{lemma}

This lemma implies that while proving Theorem \ref{tm_glavni} and considering cases from Lemma \ref{initial3}, we can simplify certain cases by reducing them to ones that have already been proven. However, there are situations where the degrees do not match, requiring us to consider these subcases separately. More precisely,
 \begin{itemize}
     \item cases \textit{2.c)} and \textit{4.c)} are reduced to case \textit{1.c)},
     \item case \textit{3.c)} is reduced to  case \textit{1.c)} but we must observe separately case $\alpha=0$ and case $2\alpha+\beta<\gamma\leq\alpha+2\beta$, 
     \item case \textit{4.a)} is reduced to  case \textit{2.a)} but we must also observe case $\gamma<2\alpha+\beta$ and, separately, case $\beta=\gamma$,
     \item case \textit{4.b)} is reduced to  case \textit{3.a)} but we must consider separately case $\beta=\gamma$.
 \end{itemize} 

We want to find all extensions of a $D(4)$-triple $\{a, b, c\}$ to a $D(4)$-quadruple $\{a, b, c, d\}$ in $\mathbb{Z}[i ][X]$. We reduce the problem of finding those extensions to
the problem of the existence of a suitable solution of equation $v_m=w_n$, where $(v_m)_{m\geq0}$ and $(w_n)_{n\geq 0}$ 
are binary recurrence sequences defined by (\ref{rekurzija_vm}) and (\ref{rekurzija_wn}), for some initial values $(z_0, x_0)$
and $(z_1, y_1)$. In Lemma \ref{initial3} we described all possible initial terms and relations between degrees of polynomials that hold in each case. We will prove that neither
of them leads to an irregular $D(4)$-quadruple with $d$ such that $\deg (d)>\gamma$. More precisely, we will show that in all cases we will have $d=d_+$ or $\deg (d)<\gamma$.

\begin{remark}\label{rem_smanjena}
    In general, for a fixed $z_0$ we have to consider two possibilities $\pm x_0$. Since we also consider the case $-z_0$, we can compare the sequence (\ref{rekurzija_vm}) with the values $(z_0,-x_0)$ and $(-z_0,x_0)$ \begin{align*} v_0&=z_0,\quad v_1=\frac{1}{2}(sz_0+c(-x_0)),\\ v'_0&=-z_0,\quad v'_1= \frac{1}{2}(s(-z_0)+cx_0)=-v_1. \end{align*} Note that if $z=v_m$ is a solution of the equation (\ref{jdba_pellova_prva}), then $-z=v'_m=-v_m$ is also a solution, but the same $d$ is obtained from (\ref{jdb_jednakosti_za_d}). Thus, it is sufficient to observe only one possibility for $x_0$, but both possibilities $\pm z_0$. The same observation holds for $z_1$ and $y_1$.
\end{remark}

\begin{proof}[Proof of Theorem \ref{tm_glavni}]
The proof is divided into parts according to the cases from Lemma \ref{initial3}.

\emph{Case 1.a)} $v_{2m}=w_{2n},\ z_0=z_1=\pm2$.\\
From (\ref{d0}), (\ref{d1}) and Remark \ref{rem_smanjena} we have $x_0= 2$ and $y_1= 2$. Inserting those values in the congruences from Lemma \ref{lemma 5.1} yields
\begin{equation}\label{kongruencija1.a}
\pm am^2+ sm\equiv \pm bn^2+ tn \ (\bmod \ c).
\end{equation}
We assume that $m,n\neq 0$, since from Proposition \ref{proposition_1_jurasic} we have that $m=0$ or $n=0$ implies $\deg(d)<\gamma$. \par
First, let $\beta<\gamma$. Congruence (\ref{kongruencija1.a}) must be an equality $\pm am^2+ sm=\pm bn^2+ tn$ since degrees of the polynomials on both sides are less than $\deg (c)=\gamma$. By observing degrees of polynomials on both sides, we get $\alpha=\beta$. Using that observation and  $\deg(v_1)=\deg(w_1)$, Lemma \ref{lema_stupnjevi} yields $m=n$. Now, we have an equality 
$$\pm m(a-b)= t+ s.$$
By multiplying it with $s-t$ and using $s^2-t^2=c(a-b)$ we get $c=\pm m(s-t)$ which cannot hold since $\deg (s- t)\leq\frac{\beta+\gamma}{2}<\gamma$.

Now, let $\beta=\gamma$. As stated in Remark \ref{Rem_dminus}, we have $d_-=0$ or $d_-=a=\pm 2i$.  
If $d_-=a=\pm 2i$ then $c=-b\pm 4i$, $v_1=\pm s+ c$, $w_1=\pm t + c$, so $\deg(b)=\deg(c)=\deg(v_1)=\deg(w_1)=\gamma$. Lemma \ref{lema_stupnjevi} now implies 
$m+\frac{1}{2}=2n,$
which cannot hold since $m$ and $n$ are integers. 

If $d_-=0$ then $c=a+b\pm 2r$, $s=\pm(r\pm a)$ and $t=\pm (b\pm r)$. 
  From (\ref{kongruencija1.a}) we see that there exists $k\in\mathbb{Z}[i]$ such that
\begin{equation}\label{kongruencija1.a_2}
\pm am^2\pm(r\pm a)m-(\pm bn^2\pm(b\pm r)n)=k(a+b\pm2r).\end{equation} 

First, let us observe the case $\alpha<\beta$, i.e.\@ $\alpha<\deg (r)<\beta$.
By comparing coefficients in (\ref{kongruencija1.a_2}) next to $b$, $r$, and $a$, respectively, we get a system of equations 
\begin{align}
     \pm n^2\pm n&=k,\label{k_prva}\\
    \pm m\pm n&=2k, \label{k_druga}\\
    \pm m^2\pm m&=k\label{k_treca}.
\end{align}
It is easy to conclude that $k\in\mathbb{Z}$ and from (\ref{k_druga}) that $m$ and $n$ must have the same parity. After observing all possibilities for signs, the only solution is $(m,n)=(2,2)$. After inserting it in expressions from Lemma \ref{lema_stupnjevi} and using that $\deg (v_1)=\deg (w_1)=\gamma$, we get $\alpha=\beta$, which is a contradiction.

It remains to consider the case $\alpha=\beta=\gamma$. Let denote by $A$, $B$ and $C=A+B\pm 2\sqrt{ AB }$ leading coefficients of $a$, $b$ and $c=a+b\pm 2r$. Also, let $S$ and $T$ denote leading coefficients of $s$ and $t$. If $v_1=\pm s+ c$ and $w_1=\pm t+ c$ are both polynomials of degree less than $\gamma$, then $S=\pm C$ and $T=\pm C$. On the other hand, $s^2=ac+4$ implies $C^2=S^2=AC$ and, similarly, $C^2=T^2=BC$. It follows that $A=B=C$ and $C=2C\pm 2 R$, where $R$ is a leading coefficient of the polynomial $r$ and $R^2= AB =C^2$ holds. It is now obvious that this is only true for $A=B=C=0$, which cannot hold.
\\So at least one of the polynomials $v_1$ and $w_1$ has a degree equal to $\gamma$. If $\deg(v_1)=\gamma$,  Lemma \ref{lema_stupnjevi} implies
$\frac{3}{2}\gamma\geq \deg(w_1)=\gamma(2m-2n+1)\geq \gamma$. The only possibility is $\deg(w_1)=\gamma$ and $2m-2n+1=1$, i.e. $m=n$. 
By inserting $m=n$ in the equation (\ref{kongruencija1.a}) we get
$$\pm m^2(a-b)\equiv \pm m(s\pm t)\ (\bmod\ c).$$
It is not hard to see from Remark \ref{Rem_dminus} that $s\pm t\in\{\pm c,\pm (a-b)\}$.
If $s\pm t=\pm c$, then $\pm m^2(a-b)\equiv 0\ (\bmod \ c)$ and if $s\pm t=\pm (a-b)$ then $\pm(m^2\pm m)(a-b)\equiv 0\ (\bmod \ c)$. Notice that $c=b-a\pm 2s$ so $a-b\equiv \pm 2s \ (\bmod \ c)$. This implies 
$$\pm M s\equiv 0 \ (\bmod \ c),$$
where $M\in \{2m^2,2m^2\pm 2m\}$ is an integer. 
From $ac+4=s^2$ we see that if a polynomial $p\in\mathbb{Z}[i][X]$ divides $c$ and $s$ then it also divides $4$, which implies that $p$ is a constant. Then, $\frac{c}{p}$, a polynomial of degree $\gamma>0$, divides a constant polynomial $M$, which holds only if $M=0$. The only possibility is $(m,n)=(1,1)$, i.e. $(2m,2n)=(2,2)$, which, by Proposition \ref{proposition_1_jurasic}, implies $d=d_+$.

\medskip
\emph{Case 1.b)} $v_{2m}=w_{2n}$, $z_0=z_1=\pm s$ and $\alpha=0$.\\
In this case $d_0=a=\pm2i$. Using (\ref{d0}), (\ref{d1}) and Remark \ref{rem_smanjena} we get $d_0=d_1=a=\pm 2i$, $x_0=0$ and $y_1=r.$ 
From Lemma \ref{lemma 5.1} we have
\begin{equation*}
    \pm asm^2\equiv \pm bsn^2+ trn\ (\bmod \ 2c).
\end{equation*}
Again, we assume $m,n\neq 0$ and multiply this congruence with $s$. After using Lemma \ref{lemma 5.2} and $s^2\equiv 4 \ (\bmod \ c)$ we get
\begin{equation}\label{kongruencija1b}
    \pm 4am^2=\pm 4bn^2+2n(a+b-d_-)\ (\bmod \ c).
\end{equation}

If $\beta<\gamma$, (\ref{kongruencija1b}) is an equality.
If $\deg(d_-)<\beta$, by comparing leading coefficients in that equation, we get $0=\pm4n^2+ 2n$, which is not possible for an integer $n\geq 1$. Hence, $\deg(d_-)=\beta$ and by Lemma \ref{degdminuslemma} we have $\deg(d_-)=\gamma+\beta$, so $\gamma=2\beta$. From Lemma \ref{lemma 3.7} we have that $d_-=-b\pm 4i=-b+2a$ or $d_-=a+b\pm 2r$. In both cases, by comparing coefficients in (\ref{kongruencija1b}), we get a contradiction with $m,n\neq 0$.

If $\beta=\gamma$, from Lemma \ref{degreesd} we have $d_-=0$ or $d_-=a=\pm 2i$. If $d_-=0$ then $c=a+b\pm 2r$. From (\ref{kongruencija1b}) there exists $k\in\mathbb{Z}[i]$, such that 
$$\pm4am^2-(\pm 4bn^2+ 2n(a+b))=k(a+b\pm 2r),$$
where we observe all possibilities for signs. Comparing leading coefficients on both sides of that equation gives
$\pm 4n^2+2n=\pm k$ and then $2k=0$ i.e. $k=0$ and $2n(\pm2n+ 1)=0$, which cannot hold for $n\geq 1$.

If $d_-=a=\pm 2i$, from Remark \ref{Rem_dminus} we have  $c=-b+2a$. 
Inserting in (\ref{kongruencija1b}) yields
$$\pm 4am^2\mp 4bn^2- 2nb=k(-b+2a),\quad k\in\mathbb{Z}[i].$$
 By comparing the coefficients we get
$\mp4n^2-2n=-k$ and $\pm 2m^2=k$, i.e. $ n(\mp2n-1)= \pm m^2$. On the other hand, since $v_1=\pm\frac{1}{2}s^2$ and $w_1=\frac{1}{2}(\pm ts\pm cr)=\frac{1}{2}(\pm (br-ar)\pm (-br+2ar))$, we have $\deg(v_1)=\gamma$ and $\deg(w_1)\in\left\{\frac{3\gamma}{2},\frac{\gamma}{2}\right\}.$ Together with Lemma \ref{lema_stupnjevi} this implies $m=2n$ or $m=2n-1$. The system of these two equations has only solution $(m,n)=(1,1)$, i.e.\@ $(2m,2n)=(2,2)$ which by Proposition \ref{proposition_1_jurasic} implies $d=d_+$.

\medskip
\emph{Case 1.c)} $v_{2m}=w_{2n},\ z_0=z_1=\pm\frac{1}{2}(cr\pm st)$ and $\alpha>0,\ \alpha+\beta\leq \gamma \leq 2\alpha+\beta$. Note that the value of $z_0$ is $\pm\frac{1}{2}(cr+st)$ or $\pm\frac{1}{2}(cr-st)$, depending on which of the two polynomials has a lower degree. 

Since $\alpha>0$, we have $\beta<\gamma$. Also, $x_0=\frac{1}{2}(at\pm rs)$ and $y_1=\frac{1}{2}(bs\pm rt)$ can be shown as, for example, in \cite{ff}.\\
First, let us observe the case $\alpha=\beta<\gamma$. Similarly as in \cite[Lemma 8]{dif2}, we have
$$\deg(v_1),\deg(w_1)\in \left\{\frac{3\gamma-\alpha}{2},\frac{\gamma+\alpha}{2}\right\}.$$
After considering all four possibilities in expressions from Lemma \ref{lema_stupnjevi}
and by comparing the degrees of the polynomials $v_{2m}$ and $w_{2n}$, we conclude that $m=n$. 
We use congruences from Lemma \ref{lemma 5.1}, and after multiplying by $st$ and using $(st)^2\equiv16\ (\bmod \ c)$, we obtain
\begin{equation}\label{eq_kong}
\pm 16(am(m\pm 1)-bn(n\pm 1))\equiv 4rst(n- m)\ (\bmod \ c).
\end{equation}
Since $m=n$, equality $16m(m\pm 1)(a-b)=0$ holds. This implies $m=n=1$. Hence, from Proposition \ref{proposition_1_jurasic} we have $d=d_+$.

Now, we observe the case $\alpha<\beta<\gamma\leq 2\alpha+\beta.$
In (\ref{eq_kong}), after applying Lemma \ref{lemma 5.2}, we see that the equality
$$\pm 16(am(m\pm 1)-bn(n\pm 1))= 8(a+b-d_-)(n- m)$$
holds.
Notice that $\deg(d_-)\leq \alpha<\beta$. By comparing leading coefficients we have $\pm 2(-n(n\pm1))=n- m$, and then
\begin{equation}\label{jdba_12}\pm 2am(m\pm 1)= (a-d_-)(n- m).\end{equation}
If $\deg (d_-)<\alpha$, then also $\pm 2m(m\pm 1)=n- m$ so $n(n\pm1)=\pm m(m\pm 1)$. Since $m$ and $n$ are positive integers, the only possibility is $m=n=1$ which implies $d=d_+$. On the other hand, if $\deg (d_-)=\alpha$, from the previous equality we have two possibilities to observe. The first possibility is $d=ka$, for some $k\in \mathbb{Z}[i]$, but then $a^2k+4=x^2$, i.e. $(a\sqrt{k}-x)(a\sqrt{k}+x)=-4$, which cannot hold for a non-constant $a$. The second possibility is that the right-hand side of (\ref{jdba_12}) vanishes, i.e.\@ $n=m$, and then $m(m\pm 1)=0$ implies $m=n=1$ and $d=d_+$. This finishes the proof of case \emph{1.c)}.\par 
\medskip
As in \cite{fj_19}, we can reduce other cases of this proof to this one by using Lemma \ref{lemma 4.4}, but it remains to prove this case for the other possibilities besides $\deg(d_-)\leq \alpha$.
The remaining cases are:
\begin{enumerate}[1.]
    \item $\alpha<\deg(d_-)<\beta$, $(2\alpha+\beta<\gamma<\alpha+2\beta)$,
    \item $\deg(d_-)=\beta$, $(\gamma=\alpha+2\beta)$,
    \item $\deg(d_-)>\beta$, $(\gamma>\alpha+2\beta)$.
\end{enumerate}
We omit the details of the proof since it is analogous as in \cite{fj_19}.

\medskip
\emph{Case 2.a)} $v_{2m+1}=w_{2n},\ (z_0,z_1)\in\{(2, s),(-2,-s)\}$ and $\gamma\geq 2\alpha+\beta$.\\
From (\ref{jdba_pellova_prva}) and (\ref{jdba_pellova_druga}) we have $x_0= 2$ and $y_1= r$. 
 
From Lemma \ref{lemma 5.1} we have
\begin{equation}\label{kong_2a_prva}
\pm asm(m+1)+ 2(2m+1)\equiv \pm bsn^2+ rtn\ (\bmod \ c).
\end{equation}
After multiplying by $s$ and using Lemma \ref{lemma 5.2}, we get 
\begin{equation}\label{kong_dokaz_2a}
     \pm4 am(m+1)+ 2s(2m+1)\equiv \pm4bn^2+ 2n(a+b-d_-)\ (\bmod \ c).
\end{equation}
\par Let us first observe the case $\beta< \gamma$. From Lemma \ref{degdminuslemma} we know $\alpha \leq \deg (d_-)<\gamma$, so we conclude that congruence (\ref{kong_dokaz_2a}) is an equality
\begin{equation}\label{jdba_dokaz_2a}
     \pm 4 am(m+1)+ 2s(2m+1)= \pm 4bn^2+2 n(a+b-d_-).
\end{equation}
Also, it is easy to see that $\deg (v_1)=\gamma$ and $\deg (w_1)\in\left\{\gamma+\frac{\alpha+\beta}{2},\gamma-\frac{\alpha+\beta}{2}\right\}$. We insert these values into expressions from Lemma \ref{lema_stupnjevi}, and $v_{2m+1}=w_{2n}$ implies
 \begin{equation}\label{eq: 2a_stup}
     2m(\alpha+\gamma)=(2n-1)(\beta+\gamma)\pm(\alpha+\beta).
 \end{equation}
We assume $m,n\geq 0$. First, the case $\beta<\frac{\alpha+\gamma}{2}=\deg (s)$ is observed. By (\ref{jdba_dokaz_2a}) we see that $\deg (d_-)=\frac{\alpha+\gamma}{2}$ and from Lemma \ref{degdminuslemma} we have $\gamma=3\alpha+2\beta$. Together with $\beta<\frac{\alpha+\gamma}{2}$ this implies $\alpha\neq 0$. Then $\deg(as)=3\alpha+\beta$, $\deg(bs)=2\alpha+2\beta$ and $\deg(rt)=2\alpha+2\beta$ are all less than $\gamma$ implying that (\ref{kong_2a_prva}) is an equality. Now observe the case $\alpha\neq \beta$. If we denote leading coefficients of polynomials $a,b,c,r,s,t$ with $A,B,C,R,S,T$, respectively, then we see that 
$$\pm BSn^2=RTn.$$
After squaring and using $R^2=AB$, $S^2=AC$ and $T^2=BC$, we get $n=1$. Lemma \ref{degineq} implies $1\leq m\leq 2$. Inserting each option in (\ref{eq: 2a_stup}) yields a contradiction with $0<\alpha\leq \beta$.  If $\alpha=\beta$, then (\ref{eq: 2a_stup}) implies $m=n-1/3$ or $m=n-2/3$, which cannot hold.
\par 
If $\beta>\frac{\alpha+\gamma}{2}$, then $2\beta-\alpha>\gamma$.  Also, by observing (\ref{jdba_dokaz_2a}), we see that  $\deg (d_-)=\beta$ must hold, so  $\gamma=2\beta+\alpha$, which implies $\alpha<0$, a contradiction. \par 
If  $\beta=\frac{\alpha+\gamma}{2}$, we have $\gamma=2\beta-\alpha$ and, since $\gamma\geq 2\alpha+\beta$, we have $\beta\geq 3\alpha$ and $\gamma\geq 5\alpha.$ Let $\alpha=0$. Then $\gamma=2\beta=2\beta+\alpha$, so we have one of the options from Lemma \ref{lemma 3.7}. If $c=r(r\pm a)(b\pm r),$ then $d_-=a+b\pm 2r$ and $s=\pm(a(b\pm r)+2)$.
We insert that in (\ref{jdba_dokaz_2a}) and get
$$
     \pm 4 am(m+1)\pm 2(ab\pm ar+2)(2m+1)= \pm 4bn^2\pm 4nr.
$$
We compare coefficients next to polynomials $b$ and $r$ and get a system
\begin{align*}
    \pm 2a(2m+1)&=4n^2,\\
    \pm 2a(2m+1)&=4n,
\end{align*}
which doesn't have solutions in positive integers $m,n,a$.
If $c=\mp 2ib^2- 8b \pm 10i,$ then $a=\pm 2i$, $d_-=-b+ 2a$ and $s=\pm(\mp2b+4i)$. Again, we insert it in (\ref{jdba_dokaz_2a}) and get 
$$
\pm 4am(m+1)+\pm2(\mp2b+4i) (2m+1)= \pm 4bn^2+2 n(2b-a).
$$
We compare coefficients next to polynomial $b$ and get $\pm(2m+1)=\pm n^2+n$, which cannot hold since $\pm n^2+n$ is an even and $2m+1$ is an odd integer.\\ 
Let $\alpha>0$. Then $\alpha\leq\deg (d_-)=\beta-2\alpha<\beta$. We use (\ref{jdba_dokaz_2a}) to define a polynomial
$$g:=nd_-+(\pm 2m(m+1)-n)a=(\pm 2n^2+n)b-(2m+1)s$$
and see that $\deg (g)\leq \deg (d_-)$.
 First, assume that $\deg (g)=\deg (d_-)=\beta-2\alpha$.
If we rewrite congruence (\ref{kong_2a_prva}) such that $s$ is collected on the left-hand side and square it, we get
\begin{equation}\label{kong_2a_kvadrirano}
    4(am(m+1)-bn^2)^2\equiv 4nr^2-4n(2m+1)rt+4(2m+1)^2 \ (\bmod \ c). 
\end{equation}
From (\ref{uvw}) and the definition of polynomial $g$, we have 
\begin{equation}\label{jednakost_2a}
    (\pm2n^2+n)rt=\pm gs\pm (2m+1)s^2\pm2(\pm 2n^2+n)v_-
\end{equation}
and $\deg (v_-)=\beta - \alpha.$ Denote $k_b=\pm 2n^2+n>0$ and $k_s=2m+1>0$. Now we multiply (\ref{kong_2a_kvadrirano}) with $k_b^2$, insert (\ref{jednakost_2a}) and use $s^2\equiv 4\ (\bmod \ c )$ and $k_bb=g+k_ss$  to get a congruence
\begin{align}\label{kong_vminus_2a}
    4k_b^2m^2&(m+1)^2a^2-8k_b^2n^2m(m+1)ab+4n^4g^2+8k_sn^4gs+16k_s^2n^4\\
    \nonumber
    &\equiv 4k_b^2nab+16k_b^2n+4k_s^2k_b^2-4k_bk_sn(\pm gs\pm 4k_s\pm2k_bv_-)\ (\bmod \ c).
\end{align}
Both sides of (\ref{kong_vminus_2a}) have a degree less than $\gamma=2\beta-\alpha$, which implies that (\ref{kong_vminus_2a}) is equality.
If $\beta>3\alpha$, then $\deg (gs)=2\beta-2\alpha>\deg (ab)$, hence, polynomial $gs$ has the highest degree in the said equality. We conclude 
$2n^3=\pm k_b=\pm (\pm2n^2+n)$, which doesn't have a solution in positive integers. Then it must hold $\beta=3\alpha$, so $\deg (gs)=\deg (ab)$.
If we compare degrees in the equality arising from (\ref{kong_vminus_2a}), we see that 
$$8k_b^2n^2m(m+1)ab+8k_sn^4gs-4k_b^2nab\pm 4k_bk_sn gs$$
must be a polynomial of degree less than or equal to $2\alpha$. 
Furthermore, since $k_ss=k_bb-g,$ we get that polynomial 
$$b(8k_b^2n^2m(m+1)a+8k_bn^4g-4k_b^2na \pm4k_b^2ng)$$
also has a degree less than $\beta$, which is only possible if the polynomial in parentheses is equal to zero. This implies that $a$ divides $g$ in $\mathbb{C}[X]$, and from the definition of $g$, we get that $a$ divides $d_-$. Hence, $d_-=\lambda a$, $\lambda\in\mathbb{C}$, which together with (\ref{uvw}) implies 
$4=(u_--\sqrt{\lambda}a)(u_-+\sqrt{\lambda}a).$
This cannot hold since at least one of the polynomials on the right-hand side is not a constant polynomial.\\
If $\deg (g)<\deg (d_-)$, then $\deg (d_-)=\alpha$ and $\beta=3\alpha$, so $\alpha>0$ and the same conclusion can be derived.

Now, it remains to observe $\beta=\gamma$. From $\gamma\geq 2\alpha+\beta$ we have $\alpha=0$. From Remark \ref{Rem_dminus} we know that $d_-=0$ or $d_-=a=\pm 2i$. If $d_-=0$, then $c=a+b\pm 2r$ and $s=\pm(a\pm r)$. From congruence (\ref{kong_dokaz_2a}) we conclude that there exists $k\in\mathbb{Z}[i]$ such that
\begin{equation}
\label{jdba_2_a_kraj}
\pm 4 am(m+1)\pm 2(a\pm r)(2m+1)\mp 4 bn^2-(2a+2b)n=k(a+b\pm 2r).
\end{equation}
By observing a leading coefficient (next to $b$)  on both sides of this equation, we get
$\mp4 n^2- 2n=k,$
which implies that $k$ is an even integer. Inserting that equality into (\ref{jdba_2_a_kraj}) leaves a new equation with a leading coefficient next to the polynomial $r$ yielding 
$\pm (2m+1)= k,$
meaning $k$ is an odd integer, which is a contradiction. In the last case, we have $d_-=a=\pm 2i$, $c=-b+2a$ and $s=\pm ir$. Again, we conclude that there exists $k\in\mathbb{Z}[i]$ such that
\begin{equation}
\label{jdba_2_a_kraj_2}
\pm 8 im(m+1)\pm 2ir(2m+1)\mp4 bn^2-2bn=k(-b\pm 4i).
\end{equation}
By observing the leading coefficients next to $b$ and $r$, we have $\mp4 n^2- 2n=-k$ and $\pm 2i(2m+1)=0$, which cannot hold.
\medskip 

We will reduce the case \textit{4.a)} for $\beta<\gamma$ to this case so we must also observe option $\gamma <2\alpha +\beta$. We have $\alpha>0$ and $\deg (d_-)<\alpha$.  If $\beta>\frac{\alpha+\gamma}{2}$, observing coefficients next to $b$ in (\ref{jdba_dokaz_2a}) implies $\pm 4n^2+2n=0$, which doesn't have a solution in positive integers. If $\beta<\frac{\alpha+\gamma}{2}$, then (\ref{jdba_dokaz_2a}) yields $2m+1=0$, which cannot hold. It remains to observe the case $\beta=\frac{\alpha+\gamma}{2}=\deg (s)$. We have $\gamma=2\beta- \alpha$, so $\alpha<\beta<3\alpha$. 
On the other hand, from (\ref{ineq20}) we get
$$\deg (z_1)=\deg (s)=\beta\leq \frac{3\gamma-\beta}{4}=\frac{5\beta-3\alpha}{4},$$
which implies $\beta\geq 3\alpha$, a contradiction.


\medskip 
\emph{Case 2.b)} $v_{2m+1}=w_{2n},\ (z_0,z_1)\in\{( s, 2),(-s,-2)\}$ and $\alpha=0$.\\
We also have $x_0=0$, $a=\pm 2i$ and $y_1=2$.  From Lemma \ref{lemma 5.1}, after using $s^2=ac+4$ and dividing by $c$ we have
\begin{equation}\label{kong_2_b}
     \pm a \pm 2 am(m+1)\equiv \pm 2bn^2+ 2tn\ (\bmod \ c).
\end{equation}
If $\beta<\gamma$ this congruence is equality so we get $n=0$ and $2m(m+1)=1$ which cannot hold. On the other hand, if $\beta=\gamma$, we have two possibilities from Remark \ref{Rem_dminus}. \\
First, we can have $d_-=0$, $c=a+b\pm 2r$ and $t=\pm(b\pm r)$. As before, we conclude that there exists $k\in\mathbb{Z}[i]$ such that congruence (\ref{kong_2_b}) gives equality
$$\pm a \pm 2 am(m+1)-(\pm 2bn^2\pm 2(b\pm r)n)=k(a+b\pm 2r).$$
By observing leading coefficients as before, we get $k=\pm 2n^2\pm 2n$ and $0=\pm 2n\pm 2k$ which cannot hold. \\
Second, we have a possibility that $a=\pm 2i$, $c=-b+2a$ and $t=\pm i(b-a)$. Similarly as before, from existence of $k\in\mathbb{Z}[i]$ such that
$$\pm a \pm 2 am(m+1)-(\pm 2bn^2\pm 2i(b-a)n)=k(-b+2a),$$
we get $k=\pm 2n^2\pm 2in$. Inserting that in the previous equality yields an equation
$$\pm 1\pm 2m(m+1)\pm 2in=\pm 4n^2\pm 4in,$$
where we observe all combinations of the signs. Now we see that $n=0$, which implies $2m(m+1)=\pm 1$ but it cannot hold for an integer $m$.

\medskip
\emph{Case 2.c)} This case can be reduced to  case \textit{1.c)} by using Lemma \ref{lemma 4.4}.

\medskip
 \emph{Case 3.a)}  $v_{2m}=w_{2n+1}$, $(z_0,z_1)\in\{(t,2),(-t,-2)\}$ and $\beta<\gamma$. 
 Also, $x_0= r$, $y_1= 2$. From Lemma \ref{lemma 5.1}, after dividing by $c$, we get
 \begin{equation}\label{kong_3a}
   \pm atm^2+ rsm\equiv \pm btn(n+1)+ 2(2n+1)\ (\bmod \ c).  
 \end{equation}
We multiply the congruence (\ref{kong_3a}) by $t$, use the fact that $t^2\equiv 4 \ (\bmod \ c)$ and 
 Lemma \ref{lemma 5.2} to get
\begin{equation}\label{kong_3a_d}
     \pm 4am^2 + 2m(a+b-d_-)\equiv \pm 4bn(n+1)+2t(2n+1)\ (\bmod \ c).
\end{equation} From Lemma \ref{degreesd} we see that $d_-\neq 0$, so Lemma \ref{degdminuslemma} implies $0\leq \deg (d_-)=\gamma-\alpha-\beta<\gamma$. On the other hand, $\beta<\deg (t)=\frac{\beta+\gamma}{2}<\gamma$  implying that congruence (\ref{kong_3a_d}) is an equality
\begin{equation}\label{eq_3a_d}
    \pm 4am^2 + 2m(a+b-d_-)= \pm 4bn(n+1)+ 2t(2n+1).
\end{equation}
We assume $m,n\neq 0$. Now we observe degrees on both sides of that equation and conclude that since $2(2n+1)\neq 0$ then $\deg (t)=\det (d_-)=\frac{\beta+\gamma}{2}$. Since $\deg (d_-)=\gamma-\alpha-\beta$, we have $\gamma=3\beta+2\alpha$. 
Then also $\deg(at)=\deg(rs)=2\alpha+2\beta< \gamma$ and $\deg(bt)=3\beta+\alpha\leq \gamma$.
First, let us observe the case $\alpha\neq \beta$ and $\alpha\neq 0$. Then $\deg (at)<\deg(bt)<\gamma$ and  congruence (\ref{kong_3a}) is an equality. It implies $n(n+1)=0$, which cannot hold for an integer $n\geq 0$.

Now it remains to observe cases $\beta>\alpha=0$ and $\alpha=\beta$.
It is easy to see that $\deg (v_1)\in\{\gamma+\frac{\alpha+\beta}{2},\gamma-\frac{\alpha+\beta}{2}\}$ and $\deg (w_1)=\gamma$. Then from Lemma \ref{lema_stupnjevi} we get 
\begin{equation}\label{3a_st}
(2m-1)(3\alpha+3\beta)\pm(\alpha+\beta)=4n(2\beta+\alpha).
\end{equation}
If $\alpha=\beta$, we get $6m-3\pm 1=6n$ which cannot hold for integers $m$ and $n$. \\ If $\alpha=0$, we get $3m=4n+1$ or $3m=4n+2$ in (\ref{3a_st}). 
Also, we have $\gamma=3\beta$ and $\deg (t)=\deg (d_-)=2\beta=\alpha+2\beta$. That means that $\{a,b,d_-\}$ is a $D(4)$-triple that satisfies conditions of Lemma \ref{lemma 3.7} and, since $c=d_+(a,b,d_-)$, we have one of the options
\begin{enumerate}[i)]
    \item $d_-=r(r\pm a)(b\pm r)$ and $t=\pm(ab^2\pm abr+3b\pm r)$, or
    \item $a=\pm 2i$, $d_-=\mp 2ib^2-8b\pm 10 i$ and $t=\pm(\mp 2b^2+5ib\pm 2)$. 
\end{enumerate} 
If case i) holds, $d_-=ab^2+4b\pm arb\pm r^3+ar^2$. Now we compare leading coefficients (next to $ab^2$) in (\ref{eq_3a_d}) and get $m=2n+1$ which is a contradiction since $m,n\geq 0$ and $3m=4n+1$ or $3m=4n+2$. In case \textit{ii)} we also compare leading coefficients (next to $b^2$) and get $im=\pm(2n+1)$ which is an obvious contradiction. 

Let us emphasise that we didn't use $\gamma \geq \alpha+2\beta $ in the proof, so we can reduce case \textit{4.b)} to this one. 
 
\medskip
 \emph{Case 3.b)} $v_{2m}=w_{2n+1}$, $z_0=\pm s,\ z_1=\pm 2$, $\alpha=0$ and $\beta=\gamma$. \\
We have $x_0^2=a^2+4$, so $x_0=0$ and $d_0=a=\pm 2i$ by (\ref{jdb_jednakosti_za_d}). Also $c=a+b\pm 2r$ and $y_1=2$.  From Lemma \ref{lemma 5.1} we have
 \begin{equation}\label{kong_3b}
 \pm2s \pm casm^2\equiv \pm2t+c(\pm btn(n+1)+2(2n+1))\ (\bmod \ c^2).
 \end{equation}
 It holds that $2(s\pm t)=2(a-b)$ or  $2(s\pm t)=\pm 2(a\pm r+b\pm r)=\pm 2c$. In the first case, there would exist $k\in\mathbb{Z}[i]$ such that $2(a-b)=k\cdot (a+b\pm 2r)$. But since $0=\alpha<\deg (r)<\beta$ we would have $k=- 2$ and $k=0$, a contradiction. In the second case, we replace  $\pm 2(s\pm t)$ with $\pm 2c$ and divide (\ref{kong_3b}) with $c$ and get a congruence 
  \begin{equation}
 \pm2 \pm asm^2\equiv \pm btn(n+1)+2(2n+1)\ (\bmod \ c).
 \end{equation}
 Since, in this case, $st\equiv \pm 4\ (\bmod \ c)$, if we multiply previous congruence with $s=\pm(a\pm r)$ we get  that equality
  \begin{equation}\label{jednakost_3b}
 \pm2(a\pm r) \pm 4am^2 \pm 4bn(n+1)\pm 2(2n+1)(a\pm r)=k(a+b\pm 2r)
 \end{equation}
 holds for a $k\in\mathbb{Z}[i]$.
 Since $0=\alpha<\deg (r)<\beta$, we compare polynomials on both sides and get a system
 \begin{align*}
     \pm 4n(n+1)&=k\\
      \pm 2\pm 2(2n+1)&= 2k\\
      \pm 2\pm 4m^2\pm 2(2n+1)&=k.
 \end{align*}
 Again we assume $m,n\neq 0$, which implies that this system doesn't have solutions in positive integers.

\medskip
 \emph{Case 3.c)} This case can be reduced to case \textit{1.c)} by using Lemma \ref{lemma 4.4}. Specially, Case \textit{3.c)} when $\alpha=0$ reduces to  case \textit{1.b)}.
 
 \medskip
 \emph{Cases 4.a), 4.b) and 4.c)} As mentioned earlier, case \textit{4.c)} can be reduced to  case \textit{1.c)}. Also case \textit{4.a)} for $\beta<\gamma$ can be reduced to \textit{2.a)} and the case \textit{4.b)} $\beta<\gamma$ can be reduced to case \textit{3.a)} It remains to prove special subcases (arising from $\beta=\gamma$) of \textit{4.a)} and \textit{4.b)}.
 
 In \textit{4.a)i)} and \textit{4.b)i)} we have $(z_0,z_1)=(\pm 2,\pm 2)$ and $z_0=z_1$, $x_0= 2$, $y_1= 2$. Also, from Lemma \ref{degsc} we know $c=a+b\pm 2r$, $s=\pm(r\pm a)$, $t=\pm(b\pm r)$. Notice that $c=\pm(t\pm s)$, i.e. $t\equiv \pm s \ (\bmod \ c)$, and also $st\equiv \pm s^2\equiv \pm 4 \ (\bmod \ c)$.  From Lemma \ref{lemma 5.1} we get
 \begin{equation}\label{eq_4_2}
 \pm2(t-s) \equiv  c(\pm (btn(n+1)-asm(m+1))+ 2(2n+1)- 2(2m+1))\ (\bmod \ c^2).
 \end{equation}
 
 Let us first observe the case $t\equiv -s \ (\bmod \ c)$.
 Since the right hand side of the congruence (\ref{eq_4_2}) is divisible by $c$, we conclude that $c$ divides $\pm2(t-s)$, i.e.
$\pm 2(t-s)\equiv 0\ (\bmod \ c),$
so $t\equiv -s \ (\bmod \ c)$ implies
$\pm 4s\equiv 0\ (\bmod \ c).$
Then from $s^2=ac+4$, we get that $c$ divides $64$, which cannot hold since $\gamma>0$.

If $t\equiv s\ (\bmod \ c)$, we have, more precisely, that $c=\pm(t-s)$, so we can divide (\ref{eq_4_2}) by $c$ and get
\begin{equation}\label{rq_4_a_b}
\pm 2\equiv \pm (btn(n+1)-asm(m+1))+ 2(2n+1)- 2(2m+1)\ (\bmod \ c).
\end{equation}
If $\alpha=\beta=\gamma$, as in case \emph{1.a)} we get $m=n$ so  we have
\begin{equation}\label{kong_4a}
\pm 2\equiv \pm n(n+1)(bt-as)\ (\bmod \ c).
\end{equation}
Observe that for $c=a+b\pm 2r$ we have $bt=ct-at\mp 2rt\equiv -at\mp 2rt\ (\bmod \ c)$, which together with $t\equiv s \ (\bmod \ c)$ and $st\equiv\pm 4 \ (\bmod \ c)$ yields
$$bt-as\equiv -a(t+s)\mp 2rt\equiv -2at\mp 2rt\equiv \pm 8\ (\bmod \ c).$$
So, from congruence (\ref{kong_4a}) we conclude that $c$ divides a constant polynomial, which is possible only for polynomial $0$, i.e. 
$\pm 2 \pm 8 n(n+1)=0.$
This equation doesn't have a solution in integers. 
\\
If $\alpha<\beta$ then $\alpha<\deg (r)<\beta$. We multiply (\ref{rq_4_a_b}) with $s$ and get
\begin{equation*}
\pm 2s\equiv \pm4 (bn(n+1)-am(m+1))+s( 2(2n+1)- 2(2m+1))\ (\bmod \ c).
\end{equation*}
Then there exists $k\in\mathbb{Z}[i]$ such that
\begin{align*}
\pm 2(r\pm a)= \pm4&(bn(n+1)-am(m+1))
\\&\pm(r\pm a)( 2(2n+1)- 2(2m+1))+k(a+b\pm 2r).
\end{align*}
When we compare leading coefficients (next to $b$), we get $k=\pm 4n(n+1)\in\mathbb{Z}$. Inserting that in the previous equation cancels the polynomial $b$ and we can compare the leading coefficients next to $r$, divide by $2$ and get
$\pm 1=\pm 2(n-m)\pm k.$
This cannot hold since there is an odd integer on the left-hand side and an even integer on the right-hand side of the equation.

 In \textit{4.a)ii)}  we observe degrees in $v_{2m+1}=w_{2n+1}$ as described in Lemma \ref{lema_stupnjevi} and get $m=2n\pm\frac{1}{2}$, which is an obvious contradiction since $m$ and $n$ are integers. 
 
 In \textit{4.b)ii)},
 from Lemma \ref{lemma 5.1} we see that 
 $\pm s^2\equiv \pm 2t\ (\bmod \ c)$.
 Notice that $t=\pm i(b-a)=\pm ic\pm 4$, hence $t\equiv 4\ (\bmod \ c)$. This yields $\pm 4\equiv \pm 8 \ (\bmod \ c)$ which implies that $c$ divides a constant polynomial, which is not possible.
 \end{proof}


For further research, one can consider the problem of the existence of $D(n)$-$m$-tuples in other rings of polynomials, such as polynomials with real or complex coefficients. In \cite{djint} it is proved that the size of a set of polynomials with complex coefficients having the property that the product of any two coefficients plus 1 is a perfect square is bounded above by $7$. However, it is not clear what the expected true upper bound is. It would be interesting to study the analogous problem for polynomial $D(4)$-$m$-tuples with real or complex coefficients and compare these results with those obtained in the case of polynomial Diophantine $m$-tuples.

 
 \section*{Acknowledgements}
The authors are supported by the Croatian Science Foundation, grant HRZZ-IP-2018-01-1313.

\bigskip 

\noindent Marija Bliznac Trebje\v{s}anin\\
{Faculty of Science \\University of Split\\ 21 000 Split\\ Croatia}
\\{marbli@pmfst.hr}
\bigskip 

\noindent Sanda Buja\v{c}i\'{c} Babi\'{c}\\
{Faculty of Mathematics\\University of Rijeka\\ 51 000 Rijeka\\ Croatia}
\\ 
{sbujacic@math.uniri.hr}



\begin{thebibliography}{9}

\bibitem{bfj} M. Bliznac Trebje\v{s}anin, A. Filipin, A. Jurasi\'{c}, {\it On the polynomial quadruples with the property $D(-1;1)$}, Tokyo J. Math. \textbf{41} (2018), 527--540.

\bibitem{moj_cetvorke} M.~Bliznac Trebješanin, {\it Extension of a Diophantine triple with the property $D(4)$}, {Acta Math.~Hungar.} \textbf{163} (2021), 213–-246.

\bibitem{cipu} N.~C.~Bonciocat,  M.~Cipu and M.~Mignotte, {\it There is no Diophantine $D(-1)$-quadruple } J. London Math. Soc., \textbf{105}, 1, (2022) 63--99.

\bibitem{duj} A.~Dujella, {\it Generalization of a problem of Diophantus}, Acta Arith. \textbf{65} (1993), 15--27.

\bibitem{dif2} A.~Dujella, C.~Fuchs, {\it Complete solution of the polynomial version of a problem of
Diophantus}, J. Number Theory \textbf{106} (2004), 326--344.


\bibitem{df} A.~Dujella, C.~Fuchs, {\it Complete solution of a problem of Diophantus and Euler}, J.
London Math. Soc. \textbf{71} (2005), 33--52.

\bibitem{dfl} A.~Dujella, C.~Fuchs, F.~Luca, {\it A polynomial variant of a problem of Diophantus for pure powers}, Int.~J.~Number Theory \textbf{4} (2008), 57--71.

\bibitem{dft} A. Dujella, C. Fuchs, R. Tichy, {\it Diophantine m-tuples for linear polynomials}, Period. Math. Hungar. \textbf{45} (2002), 21--33.

\bibitem{djint}  A. Dujella, A. Jurasi\'{c}, {\it On the size of sets in a polynomial variant of a problem of Diophantus}, Int. J. Number Theory \textbf{6} (2010), 1449--1471.

\bibitem{dujjur} A. Dujella, A. Jurasi\'{c}, \textit{On the size of sets in a polynomial variant of a problem of Diophantus},  Int. J. Number Theory (1793-0421) 6 (2010), \textbf{7}, 1449-1471

\bibitem{dl_17} A. Dujella and F. Luca, {\it On a problem of Diophantus with polynomials}, Rocky Mountain J. Math. \textbf{37} (2007), 131-157.


\bibitem{ff} A. Filipin and Y. Fujita, {\it Any polynomial D(4)-quadruple is regular}, Math. Commun. \textbf{13} (2008), 45-55.

\bibitem{fj_19} A. Filipin, A. Jurasi\'{c}, {\it A polynomial variant of a problem of Diophantus and its consequences}, Glas. Mat. Ser. III \textbf{54} (2019), 21-52.

\bibitem{glavni} A.~Filipin, A.~Jurasi\'{c}, {\it Diophantine quadruples in $\mathbb{Z}[i][X]$}, Period.~Math.~Hungar. \textbf{82} (2021), 198--212.

\bibitem{jones1} B. W. Jones, {\it A second variation of a problem of Davenport and Diophantus}, Fibonacci Quart., \textbf{16} (1977), 155-165.

\bibitem{jones2} B. W. Jones, {\it A variation of a problem of Davenport and Diophantus}, Quart. J. Math. Oxford Ser. (2), \textbf{27} (1976), 349-353.

\end{thebibliography}
\end{document}